\documentclass[11pt, a4paper]{article}
\usepackage{amsmath,amsthm,amssymb,amsbsy,upref,color,graphicx,amscd,hyperref}
\usepackage[active]{srcltx}
\usepackage[latin1]{inputenc}
\usepackage{amsmath}
\usepackage{amsfonts}
\usepackage{amssymb}
\usepackage{makeidx}
\usepackage{enumerate}

\def\ds{\displaystyle}
\def\Ss{H^{cp}}
\def\E{\mathcal{E}_X}
\def\eps{{\varepsilon}}
\def\N{\mathbb{N}}
\def\R{\mathbb{R}}
\def\A{\mathcal{A}}
\def\B{\mathcal{B}}
\def\EE{\mathcal{E}}

\def\HH{\mathcal{H}}
\def\LL{\mathcal{L}}
\def\L{\mathcal{L}}

\def\O{\Omega}

\def\<{\langle}
\def\>{\rangle}

\def\m{m}
\def\Dr{\mathcal{D}}
\newcommand{\be}{\begin{equation}}
\newcommand{\ee}{\end{equation}}
\newcommand{\bib}[4]{\bibitem{#1}{\sc#2: }{\it#3. }{#4.}}

\newcommand{\cp}{\mathop{\rm cap}\nolimits}

\numberwithin{equation}{section}
\theoremstyle{plain}
\newtheorem{teo}{Theorem}[section]
\newtheorem{lemma}[teo]{Lemma}
\newtheorem{cor}[teo]{Corollary}
\newtheorem{prop}[teo]{Proposition}
\newtheorem{deff}[teo]{Definition}
\theoremstyle{remark}
\newtheorem{oss}[teo]{Remark}

\newenvironment{ack}{{\bf Acknowledgements.}}

\title{Shape optimization problems on metric measure spaces}

\author{Giuseppe Buttazzo, Bozhidar Velichkov}

\begin{document}
\maketitle

\begin{abstract}
We consider shape optimization problems of the form
$$\min\big\{J(\O)\ :\ \O\subset X,\ m(\O)\le c\big\},$$
where $X$ is a metric measure space and $J$ is a suitable shape functional. We adapt the notions of $\gamma$-convergence and weak $\gamma$-convergence to this new general abstract setting to prove the existence of an optimal domain. Several examples are pointed out and discussed.
\end{abstract}



%%%%%%%%%%%%%%%%%%%%%%%%%%%%%%%%%%%%%%%%%%%%%%%%%%
\section{Introduction}\label{s0}

Shape optimization problems, though a classical research field starting with isoperimetric type problems, received a lot of attention from the mathematical community in the recent years, especially for their applications to Mechanics and Engineering. In particular {\it spectral optimization problems}, where one is interested in minimizing some suitable function of the spectrum of a differential operator under various types of constraints, have been widely investigated. We refer for instance to the monographs \cite{book, henrot, pierrehenrot} and to the survey paper \cite{buremc}, where the state of the art is described, together with some problems that are still open. We also refer to \cite{bubuve12} for a different type of problems, where internal obstacles are considered.

The ambient space for shape optimization problems in the literature is usually the Euclidean space $\R^d$, or sometimes a smooth Riemannian manifold (as for instance in \cite{sicbaldi}). Some examples that we illustrate in Section \ref{s6} however require a more general framework; this is for instance the case when one looks for optimal domains in a Finsler space, in a Carnot-Carath\'eodory space, or in an infinite dimensional Gaussian space.

In the present paper we consider the very general framework of metric measure spaces and we show that, under suitable conditions, spectral optimization problems admit an optimal domain as a solution. The spectrum we consider is the one of the {\it metric Laplacian}, which requires the definition of the related Sobolev spaces; the key assumption we make on the metric measure space $X$ to develop our theory is the compact embedding of the Sobolev space $H^1(X,m)$ into $L^2(X,m)$, which is satisfied in all the examples which motivated our study.

In Section \ref{s1} we recall the theory of Sobolev spaces over a metric measure space, following the approach introduced in \cite{cheeger}. In Section \ref{s2} we study boundary value problems for the metric Laplacian, together with their properties. In Section \ref{s3} we give our main existence theorem and in Section \ref{s6} we show how some interesting examples fall into our framework. Section \ref{s5} contains an abstract theory of capacity in metric measure space that could be used as an alternative approach.

%%%%%%%%%%%%%%%%%%%%%%%%%%%%%%%%%%%%%%%%%%%%%%%%%%
\section{Sobolev spaces on metric measure spaces}\label{s1}

We work in a separable metric space $(X,d)$ endowed with a finite regular Borel measure $m$ such that every open set has a non-zero measure.

\begin{deff}\label{s1d1}
Let $u:X\to\overline\R$ be a measurable function. An upper gradient $g$ for $u$ is a Borel function $g:X\to[0,+\infty]$, such that for all points $x_1,x_2\in X$ and all continuous rectifiable curves, $c:[0,l]\to X$ parametrized by arc-length, with $c(0)=x_1$, $c(l)=x_2$, we have 
$$|u(x_2)-u(x_1)|\le\int_0^l g(c(s))ds,$$
where the left hand side is intended as $+\infty$ if $|u(x_{1})|$ or $|u(x_2)|$ is $+\infty$.
\end{deff}

Following the original notation in \cite{cheeger}, for $u\in L^2(X,m)$ we set
$$|u|_{1,2}=\inf\Big\{\liminf_{j\to\infty}\|g_j\|_{L^2}\Big\},\qquad
\|u\|_{1,2}=\|u\|_{L^2}+|u|_{1,2}$$
where the infimum above is taken over all sequences $(g_j)$, for which there exists a sequence $u_j\to u$ in $L^2$ such that, for each $j$, $g_j$ is an upper gradient for $u_j$. We define the Sobolev space $H=H^1(X,m)$ as the class of functions $u\in L^2(X,m)$ such that the norm $\|u\|_{1,2}$ is finite. In \cite[Theorem 2.7]{cheeger} it was proved that the space $H^1(X,m)$, endowed with the norm $\|\cdot\|_{1,2}$, is a Banach space. Moreover, in the same work, the following notion of a gradient was introduced .

\begin{deff}\label{s1d2}
The function $g\in L^2(X,m)$ is a generalized upper gradient of $u\in L^2(X,m)$, if there exist sequences $(g_{j})_{j\ge1}\subset L^2(X,m)$ and $(u_{j})_{j\ge1}\subset L^2(X,m)$ such that
$$u_j\to u\hbox{ in }L^2(X,m),\qquad g_j\to g\hbox{ in }L^2(X,m),$$
and $g_j$ is an upper gradient for $u_j$, for every $j\ge1$.
\end{deff}

For each $u\in H^1(X,m)$ there exists a unique generalized upper gradient $g_u\in L^2(X,m)$, such that 
$$\|u\|_{1,2}=\|u\|_{L^2}+\|g_u\|_{L^2};$$
moreover, for each generalized upper gradient $g$ of $u$, we have $g_u\le g$. The function $g_u$ is called minimal generalized upper gradient. It is the metric space analogue of the modulus of the weak gradient $|\nabla u|$, when $X$ is a bounded open set of the Euclidean space and $u\in H^1(X)$, the usual Sobolev space on $X$. Moreover, under some mild conditions on the metric space $X$ and the measure $m$, the minimal generalized upper gradient has a pointwise expression (see \cite{cheeger}). In fact, for any Borel function $u$, one can define
$$Lip\,u(x)=\liminf_{r\to0}\sup_{d(x,y)=r}\frac{|u(x)-u(y)|}{r},$$
with the convention $Lip\,u(x)=0$, whenever $x$ is an isolated point. If the measure metric space $(X,d,m)$ satisfies some standard assumptions (doubling and supporting a weak Poincar\'e inequality), then the function $Lip\,u$ is the minimal generalized upper gradient (see \cite[Theorem 6.1]{cheeger}. This notion of weak differentiability is flexible enough to allow the generalization of some of the notions, typical for the calculus in the Euclidean space, to the measure metric space setting. For example, in a natural way, one can define harmonic functions, solutions of the Poisson equation on an open set and some shape functionals on the subsets $\O\subset X$ as the Dirichlet energy $E(\O)$ and the first eigenvalue of the Dirichlet Laplacian $\lambda_1(\O)$:
\begin{align}
&\ds E(\O)=\inf\Big\{\frac{1}{2}\int g_u^2\,dm(x)+\frac{1}{2}\int u^2\,dm(x)-\int u\,dm(x)\ :\nonumber\\
&\ds\hskip2truecm\ u\in L^2(X,m),\ u=0\hbox{ $m$-a.e. on }X\setminus\O\Big\},\label{s1e1}\\
&\ds\lambda_1(\O)=\inf\Big\{\frac{\int g_u^2dx}{\int u^2dx}:\ u\neq 0,\ u\in L^2(X,m),\ u=0\hbox{ $m$-a.e. on }X\setminus\O\Big\}.\label{s1e2}
\end{align}

Our main existence results concerning the functionals defined above will be proved in Sections \ref{s3} and \ref{s6}. Even if the Cheeger framework of Sobolev spaces over a metric measure space is sufficient for our purposes, we notice that the framework and the results remain valid in the following more general abstract setting.

Consider a linear subspace $H\subset L^2(X,m)$ such that:
\begin{enumerate}[(H1)]
\item $H$ is a Riesz space ($u,v\in H\ \Rightarrow\ u\vee v,u\wedge v\in H$),
\item $H$ has the Stone property ($u\in H\ \Rightarrow\ u\wedge 1\in H$).
\end{enumerate} 

Suppose that we have a mapping $D:H\to L^2(X,m)$ such that:
\begin{enumerate}[(D1)]
\item $Du\ge0$, for each $u\in H$,
\item $D(u+v)\le Du+Dv$, for each $u,v\in H$,
\item $D(\alpha u)=|\alpha|Du$, for each $u\in H$ and $\alpha\in\R$,
\item $D(u\vee v)=Du\cdot I_{\{u>v\}}+Dv\cdot I_{\{u\le v\}}$.
\end{enumerate} 

\begin{oss}\label{s1r1}
In the above hypotheses on $H$ and $D$, we have that $D(u\wedge v)=Dv\cdot I_{\{u>v\}}+Du\cdot I_{\{u\le v\}}$ and $D(|u|)=Du$. Moreover, the quantity
$$\|u\|_{H}=\left(\|u\|^2_{L^2}+\|Du\|^2_{L^2}\right)^{1/2},$$
defined for $u\in H$, is a norm on $H$ which makes the inclusion $i:H\hookrightarrow L^2$ continuous.
\end{oss}

Clearly, one can take as $H$ the Sobolev space $H^1(X,m)$ and as $Du$ the minimal generalized upper gradient $g_u$. In this case, the conditions $H1,\ H2,\ D1,\ D2,\ D3,\ D4$ are satisfied (see \cite{cheeger}).\\

Furthermore, we assume that:
\begin{enumerate}[($\HH$1)]
\item $(H,\|\cdot\|_{H})$ is complete,
\item the inclusion $i:H\hookrightarrow L^2$ is compact,
\item the norm of the gradient is l.s.c. with respect to the $L^2$ convergence, i.e. for each sequence $u_n$ bounded in $H$ and convergent in the strong $L^2$ norm to a function $u\in L^2(X,m)$, we have that $u\in H$ and
$$\int_X|Du|^2dm\le \liminf_{n\to\infty}\int_X|Du_n|^2dm.$$
\end{enumerate}

All the results that we present are valid in the general setting of a Banach space $H$ and a gradient operator $D:H\to L^2$ satisfying $H1$, $H2$, $D1$, $D2$, $D3$, $D4$, $\HH1$, $\HH2$ and $\HH3$. In fact, from now on, we will use the notation $H$ instead of $H^1(X,m)$ and $Du$ instead of $g_u$, keeping in mind that $D$ is not a linear operator. We notice that $H$ can be chosen to be any closed Riesz subspace of $H^1(X,m)$. For example, one can consider the space $H^1_0(X,m)$, defined as the closure, with respect to the norm $\|\cdot\|_{1,2}$, of the Lipschitz functions with compact support in $X$. Notice that the different choices of $H$ lead to different functionals $\lambda_1$ and $E$ (see Section \ref{s6} for more details).

%%%%%%%%%%%%%%%%%%%%%%%%%%%%%%%%%%%%%%%%%%%%%%%%%%
\section{Elliptic operators on measure metric spaces}\label{s2}

Throughout this section we will assume that $H$ is a linear subspace of $L^2(X,m)$ such that the conditions $H1$, $H2$, $D1$, $D2$, $D3$, $D4$, $\HH1$, $\HH2$ and $\HH3$ are satisfied. In the first sub-section we select a suitable set of domains on which to develop a theory of boundary value problems, analogous to the Euclidean one. Note that we do not assume that the continuous functions are dense in the Sobolev space $H$. Nevertheless, this is true for most of the choices of $H$. In those cases we can work with quasi-open sets as we will see in Section \ref{s6}.

Our definition of a Sobolev function, which has 0 as a boundary value, slightly differs from the classical one. However, from the point of view of the shape optimization problems we consider, this difference is unessential (see Theorem \ref{teoJqo}). 

\begin{deff}\label{s2d1}
For each Borel set $\O\subset X$ we define the space of Sobolev functions with zero boundary values as
$$H_{0}(\O)=\{u\in H:\ u=0\hbox{ $m$-a.e. on }X\setminus\O\}.$$
\end{deff}

We say that a function $u\in L^2(X,m)$ is a solution of the elliptic boundary value problem formally written as
\be\label{s2e1}
\begin{cases}
-\Delta u+au=f,\\
u_{|\partial\O}=0,
\end{cases}
\ee
with $f\in L^2(X,m)$ and $a>0$, if $u$ is a minimizer of the functional
$$F_\O^{a,f}(u)=\frac{1}{2}\int_X|Du|^2d\m+\frac{a}{2}\int_X|u|^2d\m-\int_X fud\m+\chi_{H_{0}(\O)}(u),$$ 
where the characteristic function $\chi_{H_{0}(\O)}$ is defined on $L^2(X,m)$ as
\be\label{s2e2}
\chi_{H_0(\O)}(u)=\begin{cases}
0&\hbox{if }u\in H_0(\O),\\
+\infty&\hbox{otherwise.}
\end{cases}
\ee
In particular, any solution of \eqref{s2e1} is in $H$.

\begin{prop}\label{s2p1}
For each Borel set $\O\subset X$, the problem \eqref{s2e1} has a unique solution $w_{\O,a,f}\in H$. Moreover, if $f\ge0$, then $w_{\O,a,f}\ge0$ $m$-a.e. on $X$.
\end{prop}

\begin{proof}
Suppose that $w_n$ is a minimizing sequence for $F_\O^{a,f}$ in $H_0(\O)$. Moreover, we can assume that for each $n>0$
$$\frac{1}{2}\int_X|Dw_n|^2d\m+\frac{a}{2}\int_X|w_n|^2d\m-\int_X fw_nd\m\le0,$$
and thus
$$\frac{1}{2}\int_X|Dw_n|^2d\m+\frac{a}{4}\int_X|w_n|^2d\m\le\frac{1}{a}\int_X f^2d\m,$$
from which we deduce that the sequence $w_n$ is bounded in $H$:
$$\|w_n\|_H\le C_a\|f\|_{L^2(X,m)}$$
for a suitable constant $C_a$. By the compact inclusion of $H$ in $L^2(X,m)$, we have that up to a subsequence 
$$w_n\xrightarrow[n\to\infty]{L^2}w\in H.$$
By the semicontinuity of the $H$ norm with respect to the convergence in $L^2(X,m)$, we have that 
$$F_{\O}^{a,f}(w)\le\liminf_{n\to\infty}F_\O^{a,f}(w_n),$$
and thus we have the existence of a minimizer. The uniqueness follows by the inequality 
$$D\left(\frac{u+v}{2}\right)\le\frac{1}{2}Du+\frac{1}{2}Dv,$$
 and the strict convexity of the $L^2$ norm. In the case when $f\ge0$, we have the inequality $F_\O^{a,f}(|u|)\le F_\O^{a,f}(u)$, for each $u\in H$ and so, by the uniqueness of the minimizer, we have that $w_{\O,a,f}\ge0$.
\end{proof}

\begin{oss}
From the proof of Proposition \ref{s2p1} we obtain, for any $f\in L^2(X,m)$ and $a>0$, the estimates
\begin{align}
&\|w_{\O,a,f}\|_H\le C_a\|f\|_{L^2(X,m)},\label{s2p1e1}\\
&\big|F_\O^{a,f}(w_{\O,a,f})\big|\le C_a\|f\|^2_{L^2(X,m)}.\label{s2p1e2}
\end{align}
\end{oss}

In the following, we will always denote with $w_{\O,a,f}$ the unique solution of \eqref{s2e1}. Since we will often consider the case $a=1$, $f=1$, we adopt the notation
\be\label{s2e3}
w_\O:=w_{\O,1,1},\qquad F_\O=F_\O^{1,1}.
\ee
For $a\in(0,+\infty)$ and $f\in L^2(X,m)$, we have comparison principles, for the family of solutions $w_{\O,a,f}$ of the problem \eqref{s2e1}, which are analogous to those in the Euclidean space. 

\begin{prop}\label{s2p2}
Assume that $f\ge0$. Then the solutions of \eqref{s2e1} satisfy the following inequalities:
\begin{enumerate}[(a)]
\item If $\omega$ and $\O$ are Borel sets in $X$ such that $\omega\subset\O$, then $w_{\omega,a,f}\le w_{\O,a,f}$.
\item If $0<a<A$, then $w_{\O,a,f}\ge w_{\O,A,f}$.
\item If $f,g\in L^2(X,m)$ are such that $f\le g$, then $w_{\O,a,f}\le w_{\O,a,g}$.
\end{enumerate}
\end{prop}

\begin{proof}
\begin{enumerate}[(a)]
\item We write, for simplicity, $u=w_{\omega,a,f}$ and $U=w_{\O,a,f}$. Consider the functions $u\vee U\in H_0(\O)$ and $u\wedge U\in H_0(\omega)$ so that
$$\chi_{H_0(\omega)}(u\wedge U)=\chi_{H_0(\omega)}(u)=0,\qquad\chi_{H_0(\O)}(u\vee U)=\chi_{H_0(\O)}(U)=0.$$
Moreover, by the minimizing property of $u$ and $U$, we have
$$F_\omega^{a,f}(u\wedge U)\ge F_\omega^{a,f}(u),$$
$$F_\O^{a,f}(u\vee U)\ge F_\O^{a,f}(U).$$
We write $X=\{u> U\}\cup\{u\le U\}$ to obtain
$$\begin{array}{ll}
&\ds\frac{1}{2}\int_{\{u> U\}\cap\omega}|DU|^2d\m+\frac{a}{2}\int_{\{u>U\}\cap\omega}|U|^2d\m-\int_{\{u>U\}\cap\omega}fU\,d\m\ge\\
&\ds\qquad\qquad\qquad\ge\frac{1}{2}\int_{\{u> U\}\cap\omega}|Du|^2d\m+\frac{a}{2}\int_{\{u> U\}\cap\omega}|u|^2d\m-\int_{\{u> U\}\cap\omega}fu\,d\m,\\
&\ds\frac{1}{2}\int_{\{u> U\}\cap\O}|Du|^2d\m+\frac{a}{2}\int_{\{u> U\}\cap\O}|u|^2d\m-\int_{\{u> U\}\cap\O}fu\,d\m\ge\\
&\ds\qquad\qquad\qquad\ge\frac{1}{2}\int_{\{u> U\}\cap\O}|DU|^2d\m+\frac{a}{2}\int_{\{u>U\}\cap\O}|U|^2d\m-\int_{\{u>U\}\cap\O}fU\,d\m.
\end{array}$$
Since $\{u>U\}\subset\omega\subset\O$, we can conclude that
$$\begin{array}{ll}
&\ds\frac{1}{2}\int_{\{u> U\}}|DU|^2d\m+\frac{a}{2}\int_{\{u> U\}}|U|^2d\m-\int_{\{u> U\}}fU\,d\m=\\
&\ds\qquad\qquad\qquad=\frac{1}{2}\int_{\{u> U\}}|Du|^2d\m+\frac{a}{2}\int_{\{u>U\}}|Du|^2d\m-\int_{\{u>U\}}fu\,d\m,
\end{array}$$
and then
$$F_\omega^{a,f}(u\wedge U)= F_\omega^{a,f}(u),$$
$$F_\O^{a,f}(u\vee U)= F_\O^{a,f}(U).$$
By the uniqueness of the minimizer, we have $u=u\wedge U$ and $U=u\vee U$. Then $u\le U$ $m$-a.e. in $X$.

\item Let $u=w_{\O,a,f}$ and $U=w_{\O,A,f}$. As before, we consider the functions $u\vee U\in H_0(\O)$ and $u\wedge U\in H_0(\O)$. We have
$$F_\O^{a,f}(u\vee U)\ge F_\O^{a,f}(u),\qquad F_\O^{A,f}(u\wedge U)\ge F_\O^{A,f}(U).$$
We write $X=\{u< U\}\cup\{u\ge U\}$ to obtain
$$\begin{array}{ll}
&\ds\frac{1}{2}\int_{\{u< U\}}|DU|^2d\m+\frac{a}{2}\int_{\{u< U\}}|U|^2d\m-\int_{\{u< U\}}fU d\m\\
&\ds\qquad\qquad\qquad\ge\frac{1}{2}\int_{\{u< U\}}|Du|^2d\m+\frac{a}{2}\int_{\{u< U\}}|u|^2d\m-\int_{\{u< U\}}fu d\m,\\
&\ds\frac{1}{2}\int_{\{u< U\}}|Du|^2d\m+\frac{A}{2}\int_{\{u< U\}}|u|^2d\m-\int_{\{u<U\}}fu d\m\\
&\ds\qquad\qquad\qquad\ge\frac{1}{2}\int_{\{u< U\}}|DU|^2d\m+\frac{A}{2}\int_{\{u< U\}}|U|^2d\m-\int_{\{u< U\}}fU d\m.
\end{array}$$
Combining the two inequalities, we have
$$\begin{array}{ll}
0&\ds\ge\left(\frac{1}{2}\int_{\{u< U\}}|Du|^2d\m+\frac{a}{2}\int_{\{u< U\}}|u|^2d\m-\int_{\{u<U\}}fu d\m\right)\\
&\ds\qquad\qquad\qquad-\left(\frac{1}{2}\int_{\{u< U\}}|DU|^2d\m+\frac{a}{2}\int_{\{u< U\}}|U|^2d\m-\int_{\{u< U\}}fU d\m\right)\\
&\ds\ge\frac{A-a}{2}\int_{\{u< U\}}(|U|^2-|u|^2)d\m\ge0.
\end{array}$$
Therefore we have, 
$$\int_{\{u< U\}}(|U|^2-|u|^2)\,d\m=0,$$
and, in conclusion, $u\ge U$ $m$-a.e. on $X$.

\item Let $u=w_{\O,a,f}$ and $U=w_{\O,a,g}$. As in the previous two points, we consider the functions $u\vee U,u\wedge U\in H_0(\omega)$, and write
$$F_\O^{a,g}(u\vee U)\ge F_\O^{a,g}(U),\qquad F_\O^{a,f}(u\wedge U)\ge F_\O^{a,f}(u).$$
We decompose the metric space $X$ as $\{u> U\}\cup\{u\le U\}$ to obtain
$$\begin{array}{ll}
&\ds\frac{1}{2}\int_{\{u> U\}}|Du|^2d\m+\frac{a}{2}\int_{\{u> U\}}u^2d\m-\int_{\{u> U\}}gu d\m\\
&\ds\qquad\qquad\qquad\ge\frac{1}{2}\int_{\{u> U\}}|DU|^2d\m+\frac{a}{2}\int_{\{u> U\}}U^2d\m-\int_{\{u> U\}}gU d\m,\\
&\ds\frac{1}{2}\int_{\{u> U\}}|DU|^2d\m+\frac{a}{2}\int_{\{u> U\}}U^2d\m-\int_{\{u>U\}}fU d\m\\
&\ds\qquad\qquad\qquad\ge\frac{1}{2}\int_{\{u> U\}}|Du|^2d\m+\frac{a}{2}\int_{\{u> U\}}u^2d\m-\int_{\{u> U\}}fu d\m.
\end{array}$$
Then, we have
$$\begin{array}{ll}
0&\ds\ge\frac{1}{2}\int_{\{u> U\}}|Du|^2d\m+\frac{a}{2}\int_{\{u> U\}}u^2d\m-\int_{\{u> U\}}fu d\m\\
&\ds\qquad\qquad\qquad-\left(\frac{1}{2}\int_{\{u> U\}}|DU|^2d\m+\frac{a}{2}\int_{\{u> U\}}U^2d\m-\int_{\{u> U\}}fU d\m\right)\\
&\ds\ge\int_{\{u> U\}}(g-f)u d\m-\int_{\{u> U\}}(g-f)U d\m=\int_{\{u> U\}}(g-f)(u-U) d\m\ge0.
\end{array}$$
Thus, we obtain the equality
$$\begin{array}{ll}
&\ds\frac{1}{2}\int_{\{u> U\}}|Du|^2d\m+\frac{a}{2}\int_{\{u> U\}}u^2d\m-\int_{\{u> U\}}fu d\m\\
&\ds\qquad\qquad\qquad=\frac{1}{2}\int_{\{u> U\}}|DU|^2d\m+\frac{a}{2}\int_{\{u> U\}}U^2d\m-\int_{\{u> U\}}fU d\m,
\end{array}$$
and, in terms of the functional $F_\O^{a,f}$,
$$F_\O^{a,f}(u)=F_\O^{a,f}(u\wedge U).$$
By the uniqueness of the minimizer of $F_\O^{a,f}$, we conclude that $U\ge u$ $m$-a.e.
\end{enumerate}
\end{proof}

We now prove a result, analogous to the strong maximum principle for elliptic operators in the Euclidean space. To prove this result we need the following lemma, which is similar to \cite[Proposition 3.1]{dmgar}.

\begin{lemma}\label{s2l1}
Fix an arbitrary $u\in H_0(\O)$ and consider the sequence of functionals defined on $L^2(X,m)$
$$F_n(v)=\frac{1}{2}\int_{X}|Dv|^2d\m+\frac{n}{2}\int_{X}|v-u|^2 d\m+\chi_{H_{0}(\O)}(v).$$
Each of these functionals has a unique minimizer $u_n\in H_0(\O)$. The sequence of minimizers $u_n$ is convergent to $u$ in strongly in $L^2(X,m)$; more precisely, we have
$$\|u_n-u\|^2_{L^2(X)}\le\frac{C}{n}.$$
\end{lemma}

\begin{proof}
For each $n\ge1$, we have
$$\frac{n}{2}\int_{X}|u_n-u|^2 d\m\le F_n(u_n)\le F_n(u)=\frac{1}{2}\int_{X}|Du|^2d\m,$$
which concludes the proof.
\end{proof}

The following proposition replaces the classical ``strong maximum principle'' in the general metric space framework.

\begin{prop}\label{s2p3}
Let $\O\subset X$ be a Borel set and let $w_\O$ be the solution of the problem \eqref{s2e1} with $a=1$ and $f=1$. Then for each $u\in H_0(\O)$, we have $\{u\neq0\}\subset\{w_\O>0\}$ $m$-a.e..
\end{prop}

\begin{proof}
Considering $|u|$ instead of $u$, we can restrict our attention only to nonnegative functions. Moreover, by taking $u\wedge 1$, we can suppose that $0\le u\le1$. Consider the family of functionals $(F_n)_{n\in\N}$, introduced in Lemma \ref{s2l1}, together with the corresponding minimizers $(u_n)_{n\in\N}$. Observe that $u_n$ is also a minimizer of the functional
$$F'_n(v)=\frac{1}{2}\int_{X}|Dv|^2d\m+\frac{n}{2}\int_{X}v^2 d\m-n\int_{X}vu d\m+\chi_{H_0(\O)}(v).$$
Consider the sequence of functionals
$$G_n(v)=\frac{1}{2}\int_{X}|Dv|^2d\m+\frac{1}{2}\int_{X}v^2d\m-n\int_{X}vu d\m+\chi_{H_0(\O)}(v);$$
each of them has a unique minimizer $v_n$. Since $n w_\O$ is the unique minimizer of the functional
$$F_\O^{(n)}(v)=\frac{1}{2}\int_X|Dv|^2d\m+\frac{1}{2}\int_{X}v^2d\m-n\int_{X}v d\m+\chi_{H_0(\O)}(v),$$
we have, by the weak maximum principle (Proposition \ref{s2p2}), that $nw_\O\ge v_n\ge u_n$. Thus, the inclusion $\{u_n>0\}\subset\{w_\O>0\}$ holds $m$-a.e. for each $n$ and, passing to the limit as $n\to\infty$, we obtain $\{w_\O>0\}\supset\{u>0\}$.
\end{proof}

\begin{cor}\label{s2c1}
Let $\O\subset X$ be a Borel set and let $w_\O$ be the solution of the problem \eqref{s2e1} with $a=1$ and $f=1$. Then, we have
\begin{enumerate}[(a)]
\item $H_0(\O)=H_0(\{w_\O>0\})$,
\item $\lambda_1(\O)=\lambda_1(\{w_\O>0\})$,
\item $E(\O)=E(\{w_\O>0\})$.
\end{enumerate}
\end{cor}

\begin{deff}\label{s2d2}
We say that the Borel set $\O\subset X$ is an energy set, if the solution $w_\O$ of \eqref{s2e1} with $a=1$ and $f=1$ is such that $m(\O\setminus\{w_\O>0\})=0$.
\end{deff}

Proposition \ref{s2p3} allows us to associate to each energy set $\O$ a unique function $w_\O\in H$. This identification will allow us to import some of the Banach space properties of $H$ into the family of energy sets. In particular, in the next section we will introduce a notion of convergence for this class of domains.

\begin{oss}\label{s2r1}
For each $u\in H$ the set $\O=\{u>0\}$ is an energy set. In fact, $\{w_\O>0\}\subset\{u>0\}$ since $w_\O\in H_0(\O)$, while for the opposite inclusion we use Proposition \ref{s2p3}, by which we have that $\{u>0\}=\{u^+\neq0\}\subset\{w_\O>0\}$. We will give a precise characterization of the energy sets in Section \ref{s5} (see Remark \ref{oss0}).
\end{oss}

\begin{oss}\label{s2r2}
Suppose that $F$ is a functional defined on the family of closed linear subspaces of $H$. If $\O$ is a solution of the shape optimization problem 
\be\label{sopF}
\min\Big\{F(H_0(\O)):\ \O\subset X,\ \O\hbox{ Borel, }m(\O)\le c\Big\},
\ee
then $\{w_\O>0\}$ is also a solution of the same problem \eqref{sopF}, i.e. there exists a solution which is an energy set. 
\end{oss}

%%%%%%%%%%%%%%%%%%%%%%%%%%%%%%%%%%%%%%%%%%%%%%%%%%
\section{The $\gamma$ and weak-$\gamma$ convergences}\label{s3}

Throughout this section we will assume that all the properties $H1,H2$, $D1,D2,D3$, $D4$, $\HH1,\HH2$ and $\HH_3$ are satisfied. We introduce a suitable topology on the class of energy sets, which allows us to prove the main existence result Theorem \ref{teoJ}.

\begin{deff}\label{s3d1}
We say that a sequence of energy sets $\O_n$ $\gamma$-converges to the energy set $\O$ if $w_{\O_n}$ converges to $w_\O$ strongly in $L^2(X,m)$.
\end{deff}

\begin{deff}\label{s3d2}
We say that a sequence of energy sets $\O_n$ weak-$\gamma$-converges to the energy set $\O$ if the sequence $\left(w_{\O_n}\right)_{n\ge1}$ is strongly convergent in $L^2(X,m)$ and its limit $w\in H$ is such that $\{w>0\}=\O$.
\end{deff}

\begin{oss}\label{s3r1}
The family of energy set is sequentially compact with respect to the weak $\gamma$ convergence. In fact, by \eqref{s2p1e1} and the compact inclusion of $H$ in $L^2(X,m)$, we have that each sequence of energy sets has a weak-$\gamma$ convergent subsequence.
\end{oss}

\begin{prop}\label{s3p1}
Suppose that a sequence of energy sets $\O_n$ weak-$\gamma$-converges to $\O$ and suppose that $(u_n)_{n\ge0}\subset H$ is a sequence bounded in $H$ and strongly convergent in $L^2(X,m)$ to a function $u\in H$. If $u_n\in H_0(\O_n)$ for every $n$, then $u\in H_0(\O)$.
\end{prop}

\begin{proof}
First, taking $|u_n|$, we can suppose $u_n\ge0$ for every $n\ge1$. Moreover, by considering the sequence $u_n\wedge 1$, we can also suppose that $0\le u_n\le1$. For each $n,k\ge1$ we define on $L^2(X,m)$ the functional
$$F_{n,k}(v)=\frac{1}{2}\int_{\O_n}|Dv|^2d\m+\frac{1}{2}\int_{\O_n}|v|^2d\m-\int_{\O_n}vd\m+\frac{k}{2}\int_{\O_n}|v-u_n|^2d\m+\chi_{H_0(\O_n)}(v).$$
Denote by $u_{n}^{k}\in H_0(\O_n)$ the (unique) minimizer of $F_{n,k}$. By Lemma \ref{s2l1}, we have that
$$u_n^k\xrightarrow[k\to\infty]{L^2(X,m)}u_n.$$
Moreover, as in the proof of Proposition \ref{s2p3}, we have the inequality $u_{n}^{k}\le (k+1)w_{\O_n}$.

Observe that, since the sequence $(w_{\O_n})_{n\ge1}$ is weakly convergent in $H$ it is also bounded in the norm of $H$. Then, for each $k\ge1$, the sequence $(u_n^k)_{n\ge1}$ is also bounded in $H(X)$ and so, by extracting a subsequence, we can suppose that $u_n^k\xrightarrow[n\to\infty]{L^2}u^k$, for some $u_k\in H(X)$. Moreover, $u^k\in H_0(\O)$, since, passing to the limit as $n\to\infty$, we have $u^k\le(k+1)w$.

By the fact that $u_n^k$ is the minimizer of $F_{n,k}$, we have
$$F_{n,k}(u_n^k)\le F_{n,k}(u_n)=\frac{1}{2}\int_{\O_n}|Du_n|^2d\m+\frac{1}{2}\int_{\O_n}|u_n|^2d\m-\int_{\O_n}u_nd\m\le C,$$
where the last inequality is due to the hypothesis that $u_n$ is bounded in $H$. On the other hand
$$\begin{array}{ll}
F_{n,k}(u_n^k)&\ds=\frac{1}{2}\int_{\O_n}|Du_n^k|^2d\m+\frac{1}{2}\int_{\O_n}|u_n^k|^2d\m-\int_{\O_n}u_n^kd\m+\frac{k}{2}\int_{\O_n}|u_n^k-u_n|^2d\m\\
&\ds\ge\frac{1}{2}\int_{\O_n}|Dw_{\O_n}|^2d\m+\frac{1}{2}\int_{\O_n}|w_{\O_n}|^2d\m-\int_{\O_n}w_{\O_n}d\m+\frac{k}{2}\int_{\O_n}|u_n^k-u_n|^2d\m.
\end{array}$$
Putting the both inequalities together, we have
$$k\int_{\O_n}|u_n^k-u_n|^2d\m\le C-F_{\O_n}(w_{\O_n})\le C',$$
where the last inequality is due to the boundedness of the sequence $(w_{\O_n})_{n\ge1}$ in $H$.
In conclusion, writing $C$ instead of $C'$, we have 
$$\int_{X}|u_n^k-u_n|^2d\m\le\frac{C}{k},$$
and passing to the limit as $n\to\infty$,
$$\int_{X}|u^k-u|^2d\m\le\frac{C}{k}.$$
Thus, we found a sequence $(u^k)_{k\ge0}\subset H_0(\O)$ convergent in $L^2(X,m)$ to $u$. Then $u\in H_0(\O)$ by Definition \ref{s2d1}.
\end{proof}

We note that Proposition \ref{s3p1} is sufficient for the proof of the existence results Theorem \ref{sopphi} and Theorem \ref{main} (see Remark \ref{s3r1} and Proposition \ref{s4p2}). In order to prove the more general Theorem \ref{teoJ}, we need a make a more careful analysis of the relations between the $\gamma$ and the weak-$\gamma$ convergences. From here to the end of this section, we present an argument which allows us to avoid the notion of capacitary measures, which in the Euclidean setting are the closure of the energy sets with respect to the $\gamma$-convergence (see \cite{dmmo}, \cite{dmgar}, \cite{budm1}, \cite{budm2}, \cite{bdmexistence}). 

\begin{lemma}\label{s3l1}
Consider a sequence $\O_n$ of energy sets, which weak-$\gamma$ converges to the energy set $\O$. Suppose that for each $n\ge1$ we have that $\O\subset\O_n$. Then $w=w_\O$.
\end{lemma}

\begin{proof}
For a given Borel set $A\subset X$, consider the sequence of functionals
$$\begin{array}{ll}i
F_{\O_n}^{(A)}(u)&\ds=\frac{1}{2}\int_{A\cap\O_n}|D(u+w_{\O_n})|^2d\m+\frac{1}{2}\int_{A\cap\O_n}|u+w_{\O_n}|^2d\m\\
&\ds\qquad\qquad\qquad-\int_{A\cap\O_n}(u+w_{\O_n})d\m+\chi_{H_0(\O_n\cap A)}(u),\\
F_{\O_n}^{(A^c)}(u)&\ds=\frac{1}{2}\int_{A^c\cap\O_n}|D(u+w_{\O_n})|^2d\m+\frac{1}{2}\int_{A^c\cap\O_n}|u+w_{\O_n}|^2d\m\\
&\ds\qquad\qquad\qquad-\int_{A^c\cap\O_n}(u+w_{\O_n})d\m+\chi_{H_0(\O_n\cap A^c)}(u),
\end{array}$$
defined on $L^2(X,m)=L^2(A,m)\oplus L^2(A^c,m)$. If $u,v\in H_0(\O_n)$, then we have
$$F_{\O_n}^{(A)}(u)+F_{\O_n}^{(A^c)}(v)=
\begin{cases}
\infty&\hbox{if $u\notin H_0(A)$ or }v\notin H_0(A^c),\\
F_{\O_n}(u+v+w_n)&\hbox{otherwise.}
\end{cases}$$
Then, by the uniqueness of the minimizer of $F_{\O_n}$, we have that $F_{\O_n}^{(A)}$ and $F_{\O_n}^{(A^c)}$ admit unique minimizers, both equal to $0$. We can assume, up to a subsequence, that there are functionals $F^{(A)}$ and $F^{(A^c)}$ on $L^2(X,m)$ such that
$$F_{\O_n}^{(A)}\xrightarrow[]{\Gamma}F^{(A)}\qquad\hbox{and}\qquad
F_{\O_n}^{(A^c)}\xrightarrow[]{\Gamma}F^{(A^c)},$$
where the $\Gamma$-convergence is for the functionals defined on the metric space $L^2(X,m)$. Observe that for each $u_n\xrightarrow[]{L^2(X,m)} u$, we have
$$\liminf_{n\to\infty}F_{\O_n}^{(A)}(u_n)\ge G^{(A)}(u),$$
where we defined
$$\begin{array}{ll}
&\ds G^{(A)}(u)=\frac{1}{2}\int_A|D(u+w)|^2d\m+\frac{1}{2}\int_A|u+w|^2d\m-\int_A(u+w)d\m+\chi_{H_0(\O\cap A)}(u),\\
\\
&\ds G^{(A^c)}(u)=\frac{1}{2}\int_{A^c}|D(u+w)|^2d\m+\frac{1}{2}\int_{A^c}|u+w|^2d\m-\int_{A^c}(u+w)d\m+\chi_{H_0(\O\cap A^c)}(u).
\end{array}$$
As a consequence, we have
$$F^{(A)}(u)\ge G^{(A)}(u)\qquad\hbox{and}\qquad
F^{(A^c)}(u)\ge G^{(A^c)}(u).$$
Suppose that $A=\{w>w_\O\}$ and $A^c=\{w\le w_\O\}$. Then the functionals $G^{(A)}$ and $G^{(A^c)}$ have unique minimizers $(w_\O-w)I_A\in H^1_0(A)$ and $(w_\O-w)I_{A^c}\in H^1_0(A^c)$. Consider the functions $(w_\O-w_{\O_n})I_{A_n}$, where $A_n=\{w_{\O_n}>w_\O\}$. Since $A_n\subset A$ and $\O\subset\O_n$, we have that $(w_\O-w_{\O_n})I_{A_n}\in H_0(\O_n)\cup H_0(A)$
$$F_{\O_n}^{(A)}\left(I_{A_n}(w_\O-w_{\O_n})\right)=\frac{1}{2}\int_{A}|D(w_\O\wedge w_{O_n})|^2d\m+\frac{1}{2}\int_{A}|w_\O\wedge w_{O_n}|^2d\m-\int_{A}w_\O\wedge w_{O_n} d\m,$$
By the definition of the $\Gamma$-limit and the fact that 
$$w_\O\wedge w_{\O_n}\xrightarrow[n\to\infty]{L^2(X,m)}w_\O\wedge w,$$ 
we have
\be
\begin{array}{ll}
F^{(A)}\left((w_\O-w)I_A\right)&\le\liminf_{n\to\infty}F^{(A)}_{\O_n}\left((w_\O-w_{\O_n})I_{A_n}\right)\\
\\
&\le\frac{1}{2}\int_{A}|D(w_\O\wedge w)|^2d\m+\frac{1}{2}\int_A|w_\O\wedge w|^2d\m-\int_{A}w_\O\wedge w d\m\\
\\
&=G^{(A)}((w_\O-w)I_A)=G^{(A)}(w_\O-w).
\end{array}
\ee

Recall that, since $w_\O-w$ is a minimizer for $G^{(A)}$, we have the opposite inequality, and so the equality
$$F^{(A)}(w_\O-w)=G^{(A)}(w_\O-w).$$
From the other side $0$ is a minimizer for $F^{(A)}$, and so we have
$$F^{(A)}(w_\O-w)\ge F^{(A)}(0)\ge G^{(A)}(0)\ge G^{(A)}(w_\O-w).$$
Thus, we obtain the following equality
$$G^{(A)}(0)=G^{(A)}(w_\O-w),$$
from which, by the uniqueness of the minimizer, the proof of the lemma is concluded.
\end{proof}

\begin{prop}\label{s3p2}
Suppose that the sequence $\O_n$ of energy sets weak-$\gamma$ converges to $\O$ and let $w=\lim_{n\rightarrow\infty}w_{\O_n}$, where the limit is strong in $L^2(X,m)$. Then we have $w\le w_\O$.
\end{prop}

\begin{proof}
For each $w_{\O_n}$ consider the energy set $\O_n^\eps=\{w_{\O_n}>\eps\}$. Then we have a simple expression for $w_{\O_n^\eps}$:
$$w_{\O_n^\eps}=(w_{\O_n}-\eps)\vee0.$$
(Otherwise, we would have a contradiction with the uniqueness of the minimizer of $F_{\O_n}$.)
But then, by the dominated convergence theorem, for a certain subsequence, we have $$w_{\O_n^\eps}\xrightarrow[n\to\infty]{}(w-\eps)\vee0.$$

By extracting a further subsequence, we can suppose 
$$w_{\O_n^\eps\cup\O}\xrightarrow[n\to\infty]{L^2}w^\eps,$$
for some $w^\eps\in H$. Moreover, again by the dominated convergence theorem, we have that $v_n^\eps\to v^\eps$ in $L^2(X)$, where
$$v^\eps_n=1-\frac{1}{\eps}(w_{\O_n}\wedge\eps),$$
$$v^\eps=1-\frac{1}{\eps}(w\wedge\eps),$$
and, as a consequence 
$$v_n^\eps\wedge w_{\O_n^\eps\cup\O}\xrightarrow[n\to\infty]{L^2(X)}v^\eps\wedge w^\eps.$$
Observe that 
$$v_n^\eps=0\hbox{ on }\O_n^\eps,\qquad w_{\O_n^\eps\cup\O}=0\hbox{ on }X\setminus(\O_n^\eps\cup\O),$$
and so we have 
$$v_n^\eps\wedge w_{\O_n^\eps\cup\O}=0\hbox{ on }X\setminus\O,$$
that is $v_n^\eps\wedge w_{\O_n^\eps\cup\O}\in H_0(\O)$ and so the limit $v^\eps\wedge w^\eps\in H_0(\O)$. Since $v^\eps=1$ on $X\setminus\O$, we have that $w^\eps\in H_0(\O)$.
Then, by the preceding lemma and the maximum principle, we have that $w^\eps\le w_\O$. From the other side, passing to the limit in the inequality $w_{\O_n^\eps}\le w_{\O_n^\eps\cup\O}$, we have
$$(w-\eps)\vee0\le w^\eps\le w_\O,$$
for each $\eps>0$. In conclusion, since $(w-\eps)\vee0\xrightarrow[\eps\to0]{L^2}w$, we have $w\le w_\O$ as required.
\end{proof}

Now we can prove the following result, which is analogous to Lemma 4.10 of \cite{buremc}.

\begin{prop}\label{s3p3}
Suppose that $(\O_n)_{n\ge1}$ is a sequence of energy sets which weak-$\gamma$-converges to the energy set $\O$. Then, there exists a sequence of energy sets $(\O'_n)_{n\ge1}$ $\gamma$-converging to $\O$ such that for each $n\ge1$ we have that $\O_n\subset\O'_n$.
\end{prop}

\begin{proof}
Consider, for each $\eps>0$, the sequence of minimizers $w_{\O_n\cup\O^\eps}$, where $\O_\eps=\{w_\O>\eps\}$.
We can suppose that for each (rational) $\eps>0$ the sequence is convergent in $L^2(X,m)$ to a positive function $w_\eps\in H$.

Consider the function $v_\eps=1-\frac{1}{\eps}(w_\O\wedge\eps)$ which is equal to $0$ on $\O_\eps$ and to $1$ on $X\setminus\O$. Then we have that $w_{\O_n\cup\O^\eps}\wedge v_\eps$ is supported on $\O_n$. Then the $L^2$-limit (which exists thanks to the dominated convergence theorem) $w_\eps\wedge v_\eps$ is supported on $\O$. In conclusion, we have that $w_\eps$ is supported on $\O$ and, by the maximum principle and the proposition above, we have $w_\eps\le w_\O$. 

From the other side, again by the maximum principle, we have $(w_\O-\eps)\vee 0\le w_\eps$. Now we can conclude by a diagonalization argument.
\end{proof}

%%%%%%%%%%%%%%%%%%%%%%%%%%%%%%%%%%%%%%%%%%%%%
\section{Functionals defined on the class of energy sets}\label{s4}

We denote with $\E$ the family of energy sets in $X$. Suppose that 
$$J:\E\to[0,+\infty],$$
is a functional on the family of energy sets such that:

\begin{enumerate}[(J1)]

\item $J$ is lower semicontinuous (shortly, l.s.c.) with respect to the $\gamma$-convergence, that is
$$J(\O)\le\liminf_{n\to\infty}J(\O_n)\qquad
\hbox{whenever }\O_n\xrightarrow{\gamma}\O.$$
\item $J$ is monotone decreasing with respect to the inclusion, that is
$$J(\O_1)\ge J(\O_2)\qquad
\hbox{whenever }\O_1\subset\O_2.$$
\end{enumerate}

\begin{lemma}\label{s4l1}
If $J$ verifies the hypothesis (J1) and (J2) above, then $J$ is l.s.c. with respect to the weak-$\gamma$-convergence.
\end{lemma}

\begin{proof}
Suppose that $\O_n\xrightarrow[n\to\infty]{weak-\gamma}\O$. By Proposition \ref{s3p3}, there exists a sequence of energy sets $(\O'_n)_{n\ge1}$ such that $\O'_n\xrightarrow[n\to\infty]{\gamma}\O$ and $\O_n\subset\O'_n$. 
Thus we have 
$$J(\O)\le\liminf_{n\to\infty}J(\O'_n)\le\liminf_{n\to\infty}J(\O_n).$$
\end{proof}

\begin{lemma}\label{s4l2}
The map $m:\O\mapsto m(\O)$, defined on the family of energy sets $\E$, is l.s.c. with respect to the weak-$\gamma$-convergence.
\end{lemma}

\begin{proof}
Consider a weak-$\gamma$ converging sequence $\O_n\xrightarrow[n\to\infty]{weak-\gamma}\O$ and the function $w\in H$ such that $\{w>0\}=\O$ and $w_{\O_n}\to w$ in $L^2(X,m)$. Up to a subsequence, we can assume that $w_{\O_n}(x)\to w(x)$ for each $x\in X$. Then by Fatou lemma
$$m(\O)=\int_X1_{\{w>0\}}dm\le\liminf_n\int_X1_{\{w_{\O_n}>0\}}dm=\liminf_n m(\O_n)$$
as required.
\end{proof}

\begin{teo}\label{teoJ}
Suppose that $J:\E\to[0,+\infty)$ is a functional on the family of energy sets which verifies the hypothesis (J1) and (J2). Then for each $c\le m(X)$, there exists an energy set $\O_c$ of measure at most $c$ which is a solution of the problem 
\be\label{sopJ}
J(\O_c)=\inf\{J(\O):m(\O)\le c,\ \O\in \E\}.
\ee
\end{teo}

\begin{proof}
Suppose that $(\O_n)_{n\ge1}$ is a minimizing sequence of energy sets of measure at most $c$. There is a weak-$\gamma$-converging subsequence which we still denote in the same way, i.e.
$$\O_n\xrightarrow[n\to\infty]{weak-\gamma}\O.$$
By Lemma \ref{s4l1} and Lemma \ref{s4l2}, we have
$$m(\O)\le \liminf_{n\to\infty}m(\O_n)\le c,$$
$$J(\O)\le\liminf_{n\to\infty}J(\O_n).$$
Then $\O$ is the desired minimizer of $J$.
\end{proof}

\begin{oss}\label{s4r1}
Let $J:\B(X)\to\R$ be a functional, defined on the family of Borel sets $\B(X)$ of $X$, of the form $J(\O)=F(H_0(\O))$, where $F$ is a functional defined on the closed linear subspaces of $H$. Then, in view of Corollary \ref{s2c1} and Remark \ref{s2r1}, $J$ is uniquely determined by its restriction on the energy sets $\E\subset\B(X)$. In particular, if $\O\in \E$ is a solution of the shape optimization problem \eqref{sopJ}, then it is also a solution of 
\be\label{s4r1e1}
\min\Big\{J(\O):\ \O\in\B(X),\ m(\O)\le c\Big\}.
\ee 
The functionals $\lambda_k$ and $E$, from Definition \ref{lb} and Definition \ref{en} below, are of this form.
\end{oss}

The main functionals we are interested in are generalizations of the first eigenvalue of the Dirichlet Laplacian and the Dirichlet energy of an open set in $\R^d$. Before we continue, we restate Definitions \ref{s1e1} and \ref{s1e2} with the generic operator $D$ instead of the generalized weak upper gradient and the space $H$ instead of the Sobolev space $H^1(X,m)$. 

\begin{deff}\label{lb}
For each Borel set $\O\in\B(X)$ the ``first eigenvalue of the Dirichlet Laplacian'' on $\O$ is defined as
\be\label{lb1}
\lambda_{1}(\O)=\inf\Big\{\int_\O|Du|^2d\m:\ u\in H_0(\O),\ \int_\O u^2d\m=1\Big\}.
\ee

More generally, we can define $\lambda_k(\O)$, for each $k>0$, as
\be\label{lbk}
\lambda_{k}(\O)=\inf_{K\subset H_0(\O)}\sup\Big\{\int_\O|Du|^2d\m:\ u\in K,\ \int_\O u^2d\m=1\Big\},
\ee
where the infimum is over all $k$-dimensional linear subspaces $K$ of $H_0(\O)$.
\end{deff}

\begin{deff}\label{en}
For each Borel set $\O\in\B(X)$ the Dirichlet Energy of $\O$ is defined as
\be\label{eneq}
E(\O)=\inf\Big\{\frac{1}{2}\int_\O|Du|^2d\m+\frac{1}{2}\int_\O u^2d\m-\int_\O ud\m:\ u\in H_0(\O)\Big\}.
\ee
\end{deff}

\begin{prop}\label{s4p1}
For each energy set $\O\subset X$ of positive measure, there is a function $u_\O\in H_0(\O)$ with $\|u_\O\|_{L^2}=1$ and such that $\int_\O|Du|^2d\m=\lambda_1(\O)$. More generally, for each $k>0$, there are functions $u_1,\dots,u_k\in H_0(\O)$ such that:
\begin{enumerate}[(a)]
\item $\|u_j\|_{L^2}=1$, for each $j=1,\dots,k$,
\item $\int_X u_iu_j=0$, for each $1\le i<j\le k$,
\item $\int_X |Du|^2dm\le \lambda_k(\O)$, for each $u=\alpha_1u_1+\dots+\alpha_ku_k$, where $\alpha_1^2+\dots+\alpha_k^2=1$. 
\end{enumerate}
\end{prop}
\begin{proof}
Suppose that $(u_n)_{n\ge1}\subset H_0(\O)$ is a minimizing sequence for $\lambda_1(\O)$ such that $\|u_n\|_{L^2}=1$. Then $(u_n)_{n\ge1}$ is bounded with respect to the norm of $H$ and so, there is a subsequence, still denoted in the same way, which strongly converges in $L^2(X,m)$ to some function $u\in H$:
$$u_n\xrightarrow[n\to\infty]{L^2(X,m)}u\in H.$$
We have that $\|u\|_{L^2}=1$ and 
$$\int_\O|Du|^2d\m\le\liminf_{n\to\infty}\int_\O|Du_n|^2d\m=\lambda_1(\O).$$
Thus, $u$ is the desired function. The proof in the case $k>1$ is analogous.
\end{proof}

\begin{prop}\label{s4p2}
For any $k>0$, consider the functional $\lambda_k:\E\to\R$ defined by \eqref{lbk}. It is decreasing with respect to the set inclusion and lower semicontinuous with respect to the weak-$\gamma$-convergence.
\end{prop}

\begin{proof}
It is clear that $\lambda_k$ is decreasing with respect to the inclusion since $\omega\subset\O$ implies $H_0(\omega)\subset H_0(\O)$. We now prove the semicontinuity. Let $\O_n\xrightarrow[n\to\infty]{w\gamma}\O$, that is $w_{\O_n}\xrightarrow[n\to\infty]{L^2(X)}w$ and $\O=\{w>0\}$. We can suppose that the sequence $\lambda_k(\O_n)$ is bounded by some positive constant $C_k$. Let for each $n>0$ the functions $u^n_1,\dots,u^n_k\in H_0(\O_n)$ satisfy the conditions $(a), (b)$ and $(c)$ from Proposition \ref{s4p1}. Then, we have that up to a subsequence we can suppose that $u^n_j$ converges in $L^2(X,m)$ to some function $u_j\in H^1(X,m)$. Moreover, by Proposition \ref{s3p1}, we have that $u_j\in H_0(\O)$, $\forall j=1,\dots,k$. Consider the linear subspace $K\subset H_0(\O)$ generated by $u_1,\dots,u_k$. Since $u_1,\dots,u_k$ are mutually orthogonal in $L^2(X,m)$, we have that $\dim K=k$ and so
$$\lambda_k(\O)\le\sup\Big\{\int_\O|Du|^2d\m:\ u\in K,\ \int_\O u^2d\m=1\Big\}.$$
It remains to prove that for each $u\in K$ such that $\|u\|_{L^2}=1$, we have
$$\int_X |Du|^2dm\le \liminf_{n\to\infty}\lambda_k(\O_n).$$
In fact, we can suppose that $u=\alpha_1u_1+\dots+\alpha_ku_k$, where $\alpha_1^2+\dots+\alpha_k^2=1$ and so, $u$ is the strong limit in $L^2(X,m)$ of the sequence $u^n=\alpha_1u^n_1+\dots+\alpha_ku^n_k\in H_0(\O_n)$. Thus, we obtain
$$\int_X |Du|^2dm\le \liminf_{n\to\infty}\int_X |Du^n|^2dm\le \liminf_{n\to\infty}\lambda_k(\O_n)$$
as required.
\end{proof}

\begin{prop}\label{s4p3}
The Dirichlet energy functional $E:\E\to\R$ introduced in Definition \ref{en}, satisfies conditions (J1) and (J2).
\end{prop}

\begin{proof}
The condition $(J1)$ is obvious in view of Definition \ref{en}. The lower semicontinuity, follows from the lower semicontinuity of the $L^2$-norm of the gradient (condition $\HH3$).
\end{proof}

In view of Proposition \ref{s4p2} and for $\lambda_k$ defined as in \eqref{lbk}, there is a large class of functionals $J$ which depend on the spectrum 
\be\label{s4e1}
\lambda(\O):=(\lambda_1(\O),\lambda_2(\O),\dots)\in\R^\N,
\ee
and which satisfy the conditions of Theorem \ref{teoJ}. In fact, consider a function 
$$\Phi:[0,+\infty]^\N\to[0,+\infty],$$
which satisfies the following conditions:

\begin{enumerate}[($\Phi$1)]
\item If $z\in[0,+\infty]^{\N}$ and $(z_n)_{n\ge1}\subset[0,+\infty]^{\N}$ is a sequence such that for each $j\in\N$
$$z_{n}^{(j)}\xrightarrow[n\to\infty]{}z^{(j)},$$
where $z_n^{(j)}$ indicates the $j^{th}$ component of $z_n$, then
$$\Phi(z)\le\liminf_{n\to\infty}\Phi(z_n).$$
\item If $z_1^{(j)}\le z_2^{(j)}$, for each $j\in\N$, then $\Phi(z_1)\le\Phi(z_2)$.
\end{enumerate}
Then, the functional $J:\E\to\R$, defined as
$$J(\O)=\Phi(\lambda(\O)),$$
satisfies the conditions $(J1)$ and $(J2)$, where for any Borel set $\O\in\B(X)$, $\lambda(\O)$ is as in \eqref{s4e1}. In particular, the shape optimization problem 
\be\label{sopphi}
\min\{\Phi(\lambda(\O))\ :\ \O\in\B(X),\ m(\O)\le c\},
\ee
has a solution. Analogously, by Proposition \ref{s4p3}, the problem
\be\label{sopen}
\min\{E(\O)\ :\ \O\in\B(X),\ m(\O)\le c\},
\ee
has a solution which is an energy set. We restate the considerations above in the following result.

\begin{teo}\label{teophi}
Suppose that $(X,d)$ is a separable metric space with a finite Borel measure $m$. Suppose that $H\subset L^2(X,m)$ and $D:H\rightarrow L^2(X,m)$ satisfy the hypothesis $H1,H2,D1,D2,D3,D4,\HH1,\HH2$ and $\HH3$. Then the shape optimization problems \eqref{sopphi} and \eqref{sopen} have solutions which are energy sets.
\end{teo}

\begin{cor}\label{corlb}
Consider a separable metric space $(X,d)$ and a finite Borel measure $m$ on $X$. Let $H^1(X,m)$ denote the Sobolev space on $(X,d,m)$ and let $Du=g_u$ be the minimal generalized upper gradient of $u\in H^1(X,m)$. Under the assumption that the inclusion $H^1(X,m)\hookrightarrow L^2(X,m)$ is compact, we have that the problems \eqref{sopphi} and \eqref{sopen} have solutions, which are energy sets. 
\end{cor}

\begin{oss}\label{s4r2}
There are various assumptions that can be made on the measure metric space $(X,d,m)$ in order to have that the inclusion $H^1(X,m)\hookrightarrow L^2(X,m)$ is compact. A detailed discussion on this topic can be found in \cite[Section 8]{smetp}. For the sake of completeness, we state here a result from \cite{smetp}:\\

Consider a separable metric space $(X,d)$ of finite diameter equipped with a finite Borel measure $m$ such that:
\begin{enumerate}[(a)]
\item there exist constants $C_m>0$ and $s>0$ such that for each ball $B(x_0,r_0)\subset X$, each $x\in B(x_0,r_0)$ and $0<r\le r_0$, we have that
$$\frac{m(B(x,r))}{m(B(x_0,r_0))}\ge C_m\frac{r^s}{r_0^s};$$
\item $(X,d,m)$ supports a weak Poincar\'e inequality, i.e. there exist $C_P>0$ and $\sigma\ge1$ such that for each $u\in H^1(X,m)$ and each ball $B=B(x,r)\subset X$ we have
$$\frac{1}{|B|}\int_{B}\left|u(y)-\frac{1}{|B|}\int_{B}udm\right|dm(y)\le C_P r\left(\frac{1}{|B(x,\sigma r)|}\int_{B(x,\sigma r)}g_u^2dm\right)^{1/2}.$$ 
\end{enumerate}
Then, the inclusion $H^1(X,m)\hookrightarrow L^2(X,m)$ is compact.
\end{oss}

%%%%%%%%%%%%%%%%%%%%%%%%%%%%%%%%%%%%%%%%%%%%%%%
\section{Quasi-open sets and energy sets}\label{s5}

In this section we introduce the notions of capacity, quasi-open sets and quasi-continuous functions in the general setting given by a linear subspace $H\subset L^2(X,m)$ and operator $D:H\mapsto L^2(X,m)$, satisfying the assumptions $H1$, $H2$, $D1$, $D2$, $D3$, $D4$, $\HH1$, $\HH2$, $\HH3$ from Section \ref{s1} and $\HH4$, defined below. All these notions are deeply studied in the Euclidean case and some of the results we present have Euclidean analogues. Thus, where possible, we will limit ourselves to state the results and give precise references (mainly, we will refer to \cite{pierrehenrot}) for their proofs in $\R^d$.

The optimization problems that appear in this setting, are the exact analogues of the ones studied in the Euclidean space. We will show that for a large class of shape functionals the results from Section \ref{s4} apply also in this context (see Theorem \ref{teoJqo}).

In order to have a capacity theory, analogous to the one in the Euclidean space, we make a further assumption on the Banach space $(H,\|\cdot\|_H)$:
\begin{enumerate}[($\HH4$)]
\item the linear subspace $H\cap C(X)$, where $C(X)$ denotes the set of real continuous functions on $X$, is dense in $H$ with respect to the norm $\|\cdot\|_{H}$.
\end{enumerate}
 
\begin{deff}\label{s5d1}
We define the capacity (that depends on $H$ and $D$) of an arbitrary set $\O\subset X$ as
\be\label{cap}
\cp(\O)=\inf\big\{\|u\|^2_H\ :\ u\in H,\ u\ge0\hbox{ on }X,\ u\ge1\hbox{ in a neighbourhood of }\O\big\}.
\ee
We will say that a property $P$ holds quasi-everywhere (shortly q.e.), if the set on which it does not hold has zero capacity.
\end{deff}

\begin{oss}\label{s5r1}
If $u\in H$ is such that $u\ge0$ on $X$ and $u\ge1$ on $\O\subset X$, then $\|u\|_H^2\le m(\O)$. Thus, we have that $\cp(\O)\ge m(\O)$ and, in particular, if the property $P$ holds q.e, then it also holds $m$-a.e.
\end{oss}

\begin{deff}\label{s5d2}
A function $u:X\to\R$ is said to be quasi-continuous if there exists a decreasing sequence of open sets $(\omega_n)_{n\ge1}$ such that:
\begin{itemize}
\item $\cp(\omega_n)\xrightarrow[n\to\infty]{}0$,
\item On the complementary $\omega_n^c$ of $\omega_n$ the function $u$ is continuous.
\end{itemize}
\end{deff}

\begin{deff}\label{s5d3}
We say that a set $\O\subset X$ is quasi-open if there exists a sequence of open sets $(\omega_n)_{n\ge1}$ such that
\begin{itemize}
\item $\O\cup\omega_n$ is open for each $n\ge1$,
\item $\cp(\omega_n)\xrightarrow[n\to\infty]{}0$.
\end{itemize}
\end{deff}

The following two Propositions contain the fundamental properties of the quasi-continuous functions and the quasi-open sets.

\begin{prop}\label{s5p1}
Suppose that a function $u:X\to\R$ is quasi-continuous. Then we have that:
\begin{enumerate}[(a)]
\item the level set $\{u>0\}$ is quasi-open,
\item if $u\ge0$ $m$-a.e., then $u\ge0$ q.e. on $X$.
\end{enumerate}
\end{prop}

\begin{proof}
See \cite[Proposition 3.3.41]{pierrehenrot} for a proof of $(a)$ and \cite[Proposition 3.3.30]{pierrehenrot} for a proof of $(b)$.
\end{proof}

\begin{prop}\label{s5t1}
\begin{enumerate}[(a)]
\item For each function $u\in H$, there is a quasi-continuous function $\tilde{u}$ such that $u=\tilde{u}$ $m$-a.e.. We say that $\tilde{u}$ is a quasi-continuous representative of $u\in H$. If $\tilde{u}$ and $\tilde{u}'$ are two quasi-continuous representatives of $u\in H$, then $\tilde{u}=\tilde{u}'$ q.e. 
\item If $u_n\xrightarrow[n\to\infty]{H}u$, then there is a subsequence $(u_{n_k})_{k\ge1}\subset H$ such that, for the quasi-continuous representatives of $u_{n_k}$ and $u$, we have
$$\tilde{u}_{n_k}(x)\xrightarrow[n\to\infty]{}\tilde{u}(x),$$
for q.e. $x\in X$.
\end{enumerate}
\end{prop}
\begin{proof}
See \cite[Theorem 3.3.29]{pierrehenrot} for a proof of $(a)$, and \cite[Proposition 3.3.33]{pierrehenrot} for a proof of $(b)$. 
\end{proof}

\begin{oss}\label{oss0}
We note that, in view of Proposition \ref{s5t1}, each energy set $\Omega$ is a quasi-open set up to a set of measure zero. In fact, by the definition of energy set, we have that $\O=\{w_\O>0\}$ m-a.e. and choosing the quasi-continuous representative of $w_\O$ we have the thesis. In the cases $H=H^1(X,m)$ and $H=H^1_0(X,m)$, we have also the converse implication. In fact, suppose that $\Omega$ is a quasi-open set and that  $\omega_n$ is a sequence of open sets, as in Definition \ref{s5d3}. For each $n\ge0$, consider the functions $w_n=w_{\O\cup\omega_n}$ and $v_n\in H$ such that $\|v_n\|_H^2\le2cap(\omega_n)$, $0\le v_n\le 1$ and $v_n\ge 1$ on $\omega_n$. Notice that, taking $u(x)=d(x,X\setminus\O)$ in Remark \ref{s2r1}, we have that $\{w_n>0\}=\O\cup\omega_n$. Consider the function $(1-v_n)\wedge w_n\in H^1_0(\Omega)$. By Proposition \ref{s2p3}, we have that $\Omega\setminus\{v_n=1\}\subset\{w_\O>0\}$ and since 
$$m(\{v_n=1\})\le \|v_n\|_H^2\le2cap(\omega_n)\to 0,$$ 
we obtain that $\{w_\O>0\}=\O$.
\end{oss}

\begin{oss}
We consider the following relations of equivalence on the Borel measurable functions 
$$u\stackrel{cp}{\sim}v,\hbox{ if }u=v\hbox{ q.e.,}\quad u\stackrel{m}{\sim}v,\hbox{ if }u=v\hbox{ $m$-a.e.}$$
We define the space
\be\label{s5c2e1}
\Ss:=\{u:X\to\R\ :\ u\hbox{ quasi-cont., }u\in H\}/\stackrel{cp}{\sim},
\ee 
and recall that 
\be\label{s5c2e11}
\Ss:=\{u:X\to\R\ :\ u\in H\}/\stackrel{m}{\sim}.
\ee 
Then the Banach spaces $\Ss$ and $H$, both endowed with the norm $\|\cdot\|_{H}$, are isomorphic. In fact, in view of Proposition \ref{s5p1} and Proposition \ref{s5t1}, it is straightforward to check that the map $[u]_{cp}\mapsto[u]_m$ is a bijection, where $[u]_{cp}$ and $[u]_m$ denote the classes of equivalence of $u$ related to $\stackrel{cp}{\sim}$ and $\stackrel{m}{\sim}$, respectively. In the sequel we will not make a distinction between $H$ and $\Ss$.
\end{oss}

For each $\O\subset X$ we define the space
\be\label{s5c1e2}
\Ss_0(\O):=\{u\in H\ :\ u=0\hbox{ q.e. on }X\setminus\O\},
\ee
which, by Theorem \ref{s5t1} $(b)$, is a closed linear subspace of $\Ss$, different from the previously defined 
\be\label{s5c2e3}
H_0(\O)=\{u\in H\ :\ u=0\hbox{ $m$-a.e. on }X\setminus\O\}.
\ee
We note that the inclusion $\Ss_0(\O)\subset H_0(\O)$ holds for each subset $\O\subset X$ and, in general, it is strict. The following Proposition explains the connection between $H_0(\O)$ and $\Ss_0(\O)$.

\begin{prop}\label{s5p2}
For each Borel set $\O$, there is a quasi-open set $\tilde{\O}$ such that:
\begin{enumerate}[(a)]
\item $\tilde\O\subset\O$ $m$-a.e. ,
\item $H_0(\O)=\Ss_0(\tilde{\O})$. 
\end{enumerate}
Moreover, if $\tilde\O$ and $\tilde\O'$ are two quasi-open sets for which $(a)$ and $(b)$ hold, then $\tilde\O=\tilde\O'$ q.e..
\end{prop}

\begin{proof}
Consider a countable dense subset $(u_k)_{k=1}^\infty=\A\subset H_0(\O)$. Then $\tilde\O$ is the desired quasi-open set, where 
$$\tilde{\O}:=\bigcup_{u\in\A}\{u\ne0\}=\{w>0\}\quad\hbox{and}\quad w=\sum_{k=1}^\infty\frac{|u_k|}{2^k\|u_k\|_H}.$$
In fact, let $u\in H_0(\O)$. Then, there is a sequence $(u_n)_{n\ge1}\subset\A$ such that $u_n\xrightarrow[n\to\infty]{H}u$ and, by Proposition \ref{s5t1} $(b)$, $u=0$ q.e. on $X\setminus\tilde{\O}$ and so, we have the first part of the thesis. Suppose that $\tilde\O=\{w>0\}$ and $\tilde\O'=\{w'>0\}$ be two quasi-open sets satisfying $(a)$ and $(b)$. Then, $w'\in H_0(\O)=\Ss_0(\tilde\O)$ and so, $\tilde\O'=\{w'>0\}\subset\tilde\O$ q.e. and analogously, $\tilde\O\subset\tilde\O'$ q.e.
\end{proof}

For some shape functionals working with energy sets or quasi-open sets makes no difference. In fact, suppose that $F$ is a decreasing functional on the family of closed linear subspaces of $H$. Then we can define the functional $J$ on the family of Borel sets, by $J(\O)=F(H_0(\O))$, and the functional $\tilde{J}$ on the class of quasi-open sets, by $\tilde{J}(\O)=F(\Ss_0(\O))$. The following result shows that the shape optimization problems with measure constraint, related to $J$ and $\tilde J$, are equivalent.

\begin{teo}\label{teoJqo}
Let $F$ be a functional on the family of closed linear spaces of $H$, which is decreasing with respect to the inclusion. Then, we have that 
\begin{align}
&\inf\big\{F(H_0(\O))\ :\ \O\hbox{ Borel, }m(\O)\le c\big\}\label{the1}\\
&\qquad=\inf\big\{F(\Ss_0(\O))\ :\ \O\hbox{ quasi-open, }m(\O)\le c\big\}.\nonumber
\end{align} 
Moreover, if one of the infima is achieved, then the other one is also achieved. 
\end{teo}

\begin{proof}
We first note that by Corollary \ref{s2c1} and Remark \ref{oss0}, the infimum in the l.h.s. of \eqref{the1} can be considered on the family of quasi-open sets. Since $F$ is a decreasing functional, we have that for each quasi-open $\O\subset X$
$$F(H_0(\O))\le F(\Ss_0(\O)).$$
On the other hand, by Proposition \ref{s5p2}, there exists $\tilde\O$ such that $m(\tilde\O)<m(\O)$ and $F(H_0(\O))=F(\Ss_0(\tilde\O))$ and so, we have that the two infima are equal.\\
\quad Let $\O_{cp}$ be a solution of the problem 
$$\min\left\{F(\Ss_0(\O))\ :\ \O\hbox{ quasi-open, }m(\O)\le c\right\}.$$
Then we have that 
$$F(H_0(\O_{cp}))\le F(\Ss_0(\O_{cp}))=\inf\left\{F(H_0(\O)):\ \O\hbox{ Borel, }m(\O)\le c\right\},$$
and so the infimum on the l.h.s. in \eqref{the1} is achieved, too.\\
\quad Let $\O_m$ be a solution of the problem 
$$\min\left\{F(H_0(\O)):\ \O\hbox{ Borel, }m(\O)\le c\right\},$$
and let $\tilde\O_m\subset\O_m$ a.e. such that $\Ss_0(\tilde\O_m)=H_0(\O_m)$. Then the infimum in the r.h.s. in \eqref{the1} is achieved in $\tilde\O_m$. In fact, we have
$$F(\Ss_0(\tilde\O_m))=F(H_0(\O_m))=\inf\left\{F(\Ss_0(\O))\ :\ \O\hbox{ quasi-open, }m(\O)\le c\right\},$$
which concludes the proof.
\end{proof}

\begin{deff}\label{lbcap}
For each quasi-open set $\O\subset X$, we define
\be\label{lbcp}
\tilde\lambda_k(\O)=\min_{K\subset\tilde H_0(\O)}\max_{0\neq u\in K}\frac{\int_\O|Du|^2dm}{\int_\O u^2dm},
\ee
where the minimum is over the $k$-dimensional subspaces $K$ of $\tilde H_0(\O)$, and 
\be\label{encp}
\tilde E(\O)=\inf\Big\{\frac{1}{2}\int_\O|Du|^2dm+\frac{1}{2}\int_\O|u|^2dm-\int_\O u\,dm:\ u\in \tilde H_0(\O)\Big\}.
\ee
\end{deff}

\begin{teo}\label{main}
In a separable metric space $(X,d)$ with a finite Borel measure $m$ on $X$, consider the linear subspace $H\subset L^2(X,m)$ and the (nonlinear) operator $D:H\to L^2(X,m)$ satisfying conditions $H1$, $H2$, $D1$, $D2$, $D3$, $D4$, $\HH1$, $\HH2$, $\HH3$ and $\HH4$. Then the shape optimization problem
\be\label{sop2}
\min\{\tilde{E}(\O)\ :\ \O\hbox{ quasi-open, }m(\O)\le c\},
\ee
has a solution. Moreover, if $\Phi:[0,+\infty]^\N\to\R$ is a function satisfying the conditions $(\Phi1)$ and $(\Phi1)$ of Section \ref{s4}, then the following shape optimization problem:
\be\label{sop1}
\min\{\Phi(\tilde{\lambda}(\O))\ :\ \O\hbox{ quasi-open, }m(\O)\le c\},
\ee
where $\tilde\lambda(\O)$ denotes the infinite vector $(\tilde\lambda_1(\O),\tilde\lambda_2(\O),\dots)\in[0,+\infty]^\N$, also has a solution.
\end{teo}

\begin{proof}
By Theorem \ref{teophi}, we have that the problem \eqref{sopphi} has a solution. Since the functionals $\lambda_k$ and $\tilde\lambda_k$ are induced by a decreasing functional on the subspaces of $H$, we can apply Theorem \ref{teoJqo} and so, problem \eqref{sop1} also has a solution. The proof that \eqref{sop2} has a solution is analogous.
\end{proof}

\begin{cor}\label{corcp}
Consider a separable metric space $(X,d)$ and a finite Borel measure $m$ on $X$. Let $H^1(X,m)$ denote the Sobolev space on $(X,d,m)$ and let $Du=g_u$ denote the minimal generalized upper gradient for any $u\in H^1(X,m)$. Suppose that $m$ is doubling and that the space $(X,d,m)$ supports a weak Poincar\'e inequality. Under the condition that the inclusion $H^1(X,m)\hookrightarrow L^2(X,m)$ is compact, we have that the problems \eqref{sop1} and \eqref{sop2} have solutions. In particular, if $X$ has finite diameter, then \eqref{sop1} and \eqref{sop2} have solutions. 
\end{cor}

\begin{proof}
Since $m$ is doubling and $(X,d,m)$ supports a weak Poincar\'e inequality of type $(1,2)$, we can apply \cite[Theorem 4.24]{cheeger}. Thus we have that the locally Lipschitz functions are dense in $H^1(X,m)$ and so, condition $\HH4$ is satisfied. Now the existence is a consequence of Theorem \ref{main}. The last claim follows from Remark \ref{s4r2}.
\end{proof}

%%%%%%%%%%%%%%%%%%%%%%%%%%%%%%%%%%%%%%%%%%%%%%%%%%
\section{Applications and examples}\label{s6}

In the previous sections we developed a general theory which allows us to threat shape optimization problems in a large class of metric spaces. Theorem \ref{main} provides a solution of the problem
$$\min\Big\{\Phi(\tilde\lambda(\O))\ :\ \O\subset X,\ \O\hbox{ quasi-open, }m(\O)\le c\Big\},$$ 
where $\Phi$ is a suitable function (see assumptions $(\Phi1)$ and $(\Phi2)$ in Section \ref{s4}), $\tilde\lambda(\O)$ is defined through a (non-linear) gradient-like functional (see Section \ref{s5}) and $c>0$. In this section we apply this result to various situations. We start discussing the classical problem when $X$ is a domain in $\R^d$ and continue with examples concerning more complex structures as Finsler manifolds, Carnot-Caratheodory spaces and infinite dimensional Hilbert spaces with Gaussian measures. We notice that for a fixed ambient space $(X,d,m)$, the shape functionals we consider depend on the choice of the space $H$. In fact, even in the case of a regular domain $X\subset\R^d$, we have that if $H=H^1_0(X)$, then $\tilde\lambda_1$ is the classical first eigenvalue of the Dirichlet Laplacian, while if $H=H^1(X)$, then $\tilde\lambda_1$, as defined in \eqref{lbcap}, is the first eigenvalue of the Laplacian with mixed Dirichlet-Neumann boundary conditions (see Section \ref{s61}). In order to distinguish these, and other similar situations, we work with the following notation:
\be\label{s6e1}
\lambda_k(\O;H):=\tilde\lambda_k(\O),\qquad E(\O;H)=\tilde E(\O),
\ee
where $\tilde\lambda_k$ and $\tilde E$ are as in Definition \ref{lbcap} and $\O$ is a quasi-open set (see Definition \ref{s5d3}).
We also adopt the notation 
\be\label{s6e3}
\lambda(\O;H)=(\lambda_1(\O;H),\lambda_1(\O;H),\dots)\in [0,+\infty]^\N.
\ee

%%%%%%%%%%%%%%%%%%%%%%%%%%%%%%%%%%%%%%%%%%%%%%%%%%
\subsection{Bounded domains in $\R^d$}\label{s61}

Consider a bounded open set $\Dr\subset\R^d$. Let $H$ be the Sobolev space $H^1_0(\Dr)\subset L^2(\Dr)$ and $D$ the Euclidean norm of the weak gradient, that is $Du=|\nabla u|$. Then, for any quasi-open set $\O\subset\Dr$, the space $\tilde H_0(\O)$ is the usual Sobolev space $H^1_0(\O)$ and $\lambda_k(\O;H)$, defined in \eqref{s6e1}, is the $k$-th eigenvalue of the Dirichlet Laplacian $-\Delta$ on $\O$. In view of the general existence result Theorem \ref{main}, we have the following:

\begin{teo}\label{s61t1}
Consider $\Dr\subset\R^d$ a bounded open set of the Euclidean space $\R^d$ and suppose that $\Phi:\R^\N\to[0,+\infty]$ satisfies conditions $(\Phi1)$ and $(\Phi2)$ from Section \ref{s4}. Then the optimization problem
$$\min\Big\{\Phi(\lambda(\O))\ :\ \O\subset\Dr,\ \O\hbox{ quasi-open, }|\O|\le c\Big\},$$
admits at least one solution, where $\lambda(\O)$ is defined in \eqref{s6e3}.
\end{teo}

\begin{oss}\label{s61r2}
A different situation occurs if one considers the functional $\lambda_k(\O;H)$ with $H=H^1(\Dr)$. In fact, if we take $\O$ and $\Dr$ regular and $k=1$, then $\lambda_1(\O;H)$ is the first eigenvalue of the Laplacian with Dirichlet condition on $\partial\O\setminus\partial\Dr$ and Neumann condition on $\partial\O\cap\partial\Dr$. With some mild regularity assumptions on $\Dr$ (for instance, $\Dr$ Lipschitz), which imply the compact inclusion of $H^1(\Dr)$ in $L^2(\Dr)$, we obtain that the problem
\be\label{s61e1}
\min\Big\{\Phi(\lambda(\O;\Dr))\ :\ \O\subset\Dr,\ \O\hbox{ quasi-open, }|\O|\le c\Big\},
\ee
has a solution, where $\Phi$ satisfies the assumptions $(\Phi1)$ and $(\Phi2)$ and $\lambda(\O;\Dr):=\lambda(\O;H^1(\Dr))$ is defined as in \eqref{s6e3}.
\end{oss}

\begin{oss}\label{s61r3}
Suppose that $d=2$, $\Dr=(0,1)\times(0,1)$ and 
$$H=\Big\{u\in H^1(\Dr): u(\cdot,0)=u(\cdot,1),\ u(0,\cdot)=u(1,\cdot)\Big\}.$$ 
For any quasi-open set $\O\subset\R^d$, which is $(1,0)$ and $(0,1)$-periodic, i.e. invariant with respect to the translations of $\R^d$ along the vectors $(1,0)$ and $(0,1)$, we define 
$$\lambda_{k,per}(\O):=\lambda_k(\O\cap\Dr;H).$$
In view of the general existence Theorem \ref{main}, we have that the problem:
\be\label{s61e2}
\min\Big\{\Phi(\lambda_{per}(\O))\ :\ \O\subset\R^d,\ \O\hbox{ quasi-open and periodic, }|\O\cap\Dr|\le c\Big\},
\ee
has a solution, where as always, $\Phi$ satisfies the assumptions $(\Phi1)$ and $(\Phi2)$ and $\lambda_{per}(\O;\Dr):=\lambda(\O\cap\Dr;H)$ is as in \eqref{s6e3}. 
\end{oss}

%%%%%%%%%%%%%%%%%%%%%%%%%%%%%%%%%%%%%%%%%%%%%%%%%%
\subsection{Finsler manifolds}

Consider a differentiable manifold $M$ of dimension $d$ endowed with a Finsler structure (for a detailed introduction to the topic see \cite{finsler}), i.e. with a map $F:TM\rightarrow[0,+\infty)$ which has the following properties:
\begin{enumerate}
\item $F$ is $C^\infty$ on $TM\setminus 0$,
\item $F$ is absolutely homogeneous, i.e. $F(x,\lambda X)=|\lambda|F(x,X)$, $\forall \lambda\in\R$,
\item $F$ is strictly convex, i.e. the Hessian matrix $g_{ij}(x)=\frac{1}{2}\frac{\partial^2}{\partial X^i\partial X^j}[F^2](x,X)$ is positive definite for each $(x,X)\in TM$.
\end{enumerate}
With these properties, $F(x,\cdot):T_x M\rightarrow [0,+\infty)$ is a norm for each $x\in M$. Writing each tangent vector in the base $(\frac{\partial}{\partial x^1},\dots,\frac{\partial}{\partial x^d})$, induced by a local coordinate chart, we obtain an isomorphism between $\R^d$ and $T_x M$ and so, we can consider the dual norm $F^\ast$ with respect to the standard scalar product on $R^d$. We define the gradient of a function $f\in C^\infty(M)$ as $Df(x):=F^\ast(x,df_x)$, where $df_x$ stays for the differential of $f$ in the point $x\in M$. The Finsler manifold $(M,F)$ is also a metric space with the distance:
$$d_F(x,y)=\inf\Big\{\int_0^1 F(\gamma(t),\dot\gamma(t))\,dt\ :\ \gamma(0)=x,\ \gamma(1)=y\Big\}.$$
For any finite Borel measure $\mu$ on $M$, we define $H^1_0(M,F,\mu)$ as the closure of the set of differentiable functions with compact support $C^\infty_c(M)$, with respect to the norm 
$$\|u\|:=\sqrt{\|u\|_{L^2(\mu)}^2+\|Du\|_{L^2(\mu)}^2}.$$
The functionals $\lambda_k$, $E$ and $\lambda$ are defined as in \eqref{s6e1} and \eqref{s6e3}, on the class of quasi-open sets, related to the $H^1(M,F,\mu)$ capacity. Various choices for the measure $\mu$ are available, according to the nature of the Finsler manifold $M$. For example, if $M$ is an open subset of $\R^d$, it is natural to consider the Lebesgue measure $\mu=\L^d$. In this case, the non-linear operator associated to the functional $\int F^\ast(x,du_x)^2dx$ is called Finsler Laplacian. On the other hand, for a generic manifold $M$, a canonical choice for $\mu$ is the Busemann-Hausdorff measure $\mu_F$, given by the volume form 
$$\frac{|B_1(0)|}{|I_x|}dx^1\wedge\dots\wedge dx^d,$$
where $B_1(0)$ is the unit ball in $\R^d$ with respect to the Euclidean distance, 
$$I_x=\{X\in T_x M:F(x,X)\le 1\},$$ 
and $|\cdot|$ is the Lebesgue measure. The Busemann-Hausdorff measure $\mu_F$ is the $d$-Hausdorff measure with respect to the distance $d_F$. The non-linear operator associated to the functional $\int F^\ast(x,du_x)^2d\mu_F(x)$ is the generalisation of the Laplace-Beltrami operator. In view of Section \ref{s5}, we have existence results for the shape functionals depending on the spectrum $\lambda(\O):=\lambda(\O;H^1(M,F,\mu))$ (see \ref{s6e3}) and related to the Finsler Laplacian and the generalized Laplace-Beltrami operators.

\begin{teo}\label{fint}
Given a compact Finsler manifold $(M,F)$ with Busemann-Hausdorff measure $\mu_F$ and a functional $\Phi:[0,+\infty]^\N\to\R$ satisfying the assumptions $(\Phi_1)$ and $(\Phi_2)$, the following problems have solutions:
$$\min\Big\{\Phi(\lambda(\O))\ :\ \mu_F(\O)\le c,\ \O\hbox{ quasi-open, }\O\subset M\Big\},$$
$$\min\Big\{E(\O)\ :\ \mu_F(\O)\le c,\ \O\hbox{ quasi-open, }\O\subset M\Big\},$$
for any fixed $c\le \mu_F(M)$.
\end{teo}

\begin{proof}
It is easy to see that the conditions $H1,H2$, $D1,D2,D3,D4$, $\HH1,\HH3$ and $\HH_4$ are satisfied for the space $H^1(M,F,\mu_F)$ and the Finsler gradient $D$. It remains to prove the compact inclusion (condition $\HH2$). It is a direct consequence of a general result for metric measure spaces (see Remark \ref{s4r2}), but it can also be obtained with a standard partition of unity argument. In fact, let $U$ be a coordinate neighbourhood centered in $x\in M$. Since $F(y,\cdot):T_yM\to\R$ is a norm for each $y\in U$, we have that there exist positive constants $c_y$ and $C_y$ such that
$$c_y|Y|\le F(y,Y)\le C_y |Y|,\ \forall Y\in T_yM,$$
where $Y=\sum_j\frac{\partial}{\partial x^j}Y^j$ and $|Y|=\sqrt{\sum_{j=1}^d|Y^j|^2}$.
Moreover, the constants $c_y$ and $C_y$ are continuous in $y$ and so, there is a coordinate neighbourhood $U_x$ centered in $x$ and positive constants $c_x, C_x$ such that
\be\label{fin}
c_x|Y|\le F(y,Y)\le C_x |Y|,\ \forall Y\in T_yM,\ y\in U_x.
\ee
Let $U_{x_k}$, $k=1,\dots,m$ be a finite cover of $M$ with coordinate neighbourhoods with constants $c_k$ and $C_k$ for which \eqref{fin} holds. Let $\phi_k$ be a partition of unity on $M$ such that $supp(\phi_k)\subset U_{x_k}$. Then the norm $\|u\|_F$ is equivalent to the norm $\sum_k\|\phi_ku\|_F$. But $\phi_k u$ has a support in $U_{x_k}$ (i.e. $u\in H^1_0(U_{x_k},F)$) and the estimate \eqref{fin} gives us the compact inclusion of each $H^1_0(U_{x_k},F)$ in $L^2$. Thus, we obtain that the inclusion of $H^1(M,F)$ in $L^2$ is compact. Applying Theorem \ref{main}, we have the thesis.
\end{proof}

\begin{teo}\label{finlap}
Consider an open set $M\subset\R^d$ endowed with a Finsler structure $F$ and the Lebesgue measure $\L^d$. If the diameter of $M$ with respect to the Finsler metric $d_F$ is finite, then the following problems have solutions:
$$\min\Big\{\Phi(\lambda(\O))\ :\ \mu_F(\O)\le c,\ \O\hbox{ quasi-open, }\O\subset M\Big\},$$
$$\min\{E(\O)\ :\ |\O|\le c,\ \O\hbox{ quasi-open, }\O\subset M\},$$
where $\Phi:[0,+\infty]^\N\to\R$ satisfies assumptions $(\Phi_1)$ and $(\Phi_2)$, $|\O|$ denotes the Lebesgue measure of $\O$ and $c\le |M|$.
\end{teo}

\begin{proof}
We first observe that the Sobolev space $H^1_0(\O,F,\L^d)$ is in fact the space $H^1(M,\LL^d)$ as defined on the measure metric space $(M,d_F,\L^d)$. Moreover, the Finsler norm of the gradient of $u$ is precisely the upper gradient $g_u$ of $u$ with respect to the metric $d_F$. To conclude, it is enough to apply Corollary \ref{corcp}.
\end{proof}

\begin{oss}
In the hypotheses of Theorem \ref{finlap} and with the additional assumption that $F$ does not depend on $x\in M$, we can apply the symmetrization technique from \cite{kawohl} to obtain that, when $c>0$ is small enough, the optimal set for the problem is a ball (with respect to the distance $d_F$) of Lebesgue measure $c$. On the other hand, if we consider a Riemannian manifold $(M,g)$ in Theorem \ref{fint}, i.e. $F(x,X)=\sqrt{g_{ij}(x)X^iX^j}$, the optimal sets for $\lambda_1$, of measure $c$, are asymptotically close to geodesic balls as $c\to0$ (see \cite{sicbaldi} for a precise statement and hypotheses on $M$). We do not know if an analogous result holds for a generic Finsler manifold.
\end{oss}

%%%%%%%%%%%%%%%%%%%%%%%%%%%%%%%%%%%%%%%%%%%%%%%%%%
\subsection{Hilbert spaces with Gaussian measure}

Consider a separable Hilbert space $(\HH,\<\cdot,\cdot\>_{\HH})$ with an orthonormal base $(e_k)_{k\in\N}$. Suppose that $\mu=N_{Q}$ is a Gaussian measure on $\HH$ with mean $0$ and covariance operator $Q$ (positive, of trace class) such that 
$$Qe_k=\lambda_k e_k,$$ 
where $0<\lambda_1\le\lambda_2\le\dots\le\lambda_n\le\dots$ is the spectrum of $Q$.

Denote with $\EE(\HH)$ the space of all linear combinations of the functions on $\HH$ which have the form $E_h(x)=e^{i\<h,x\>}$ for some $h\in\HH$. Then, the linear operator 
$$\nabla:\EE(\HH)\subset L^2(\HH,\mu)\to L^2(\HH,\mu;\HH),\qquad\qquad\nabla E_h=ih E_h,$$
is closable. We define the Sobolev space $W^{1,2}(\HH)$ as the domain of the closure of $\nabla$. Thus, for any function $u\in W^{1,2}(\HH)$, we defined the gradient $\nabla u\in L^2(\HH,\mu;\HH)$.

We denote with $\nabla_k u\in L^2(\HH,\mu)$ the components of the gradient in $W^{1,2}(\HH)$
$$\nabla_k u=\<\nabla u,e_k\>_{\HH}.$$
We have the following integration by parts formula:
$$\int_{\HH}\nabla_k uv\,d\mu+\int_{\HH}u\nabla_k v\,d\mu
=\frac{1}{\lambda_k}\int_{\HH}x_k uv\,d\mu.$$
If $\nabla_k u\in W^{1,2}(\HH)$, we have
$$\int_{\HH}\nabla_k (\nabla_k u) v\,d\mu+\int_{\HH}\nabla_k u \nabla_k v\,d\mu=\frac{1}{\lambda_k}\int_{\HH}x_k\nabla_k uv\,d\mu,$$
$$-\int_{\HH}\nabla_k (\nabla_k u) v\,d\mu-\frac{1}{\lambda_k}\int_{\HH}x_k\nabla_k uv\,d\mu=\int_{\HH}\nabla_k u\nabla_k v\,d\mu,$$
and summing (formally) over $k\in\N$, we obtain
$$\int_{\HH}(-Tr[\nabla^2 u]-\<Q^{-1}x,\nabla u\>_{\HH})v\,d\mu
=\int_{\HH}\<\nabla u,\nabla v\>_{\HH}\,d\mu,$$
where $\<Q^{-1}x,\nabla u\>_{\HH}:=\sum_{k}\frac{1}{\lambda_k}x_k\nabla_k u$.

Suppose now, that $\O\in\B(\HH)$ is a Borel set. Then we have the following

\begin{deff}
Given $\lambda\in\R$, we say that $u\in H_0(\O)=W^{1,2}_0(\O)$ is a weak solution of the equation
$$\begin{cases}
-Tr[\nabla^2 u]-\<Q^{-1}x,\nabla u\>=\lambda u,\\
u\in H_0(\O),
\end{cases}$$
if for each $v\in W^{1,2}_0(\O)$, we have
$$\int_{\HH}\<\nabla u,\nabla v\>_{\HH}\,d\mu=\lambda \int_{\HH}uvd\mu.$$
\end{deff}

By a general theorem (see \cite{davies}), we know that there is a self-adjoint operator $A$ on $L^2(\O,\mu)$ such that for each $u,v\in Dom(A)\subset W^{1,2}_0(\O)$,
$$=\int_{\HH}Au\cdot vd\mu=\int_{\HH}\<\nabla u,\nabla v\>_{\HH}d\mu.$$
Then, by the compactness of the embedding $W^{1,2}_0(\O)\hookrightarrow L^2(\mu)$, $A$ is a positive operator with compact resolvent. Keeping in mind the construction of $A$, we will write 
$$A=-Tr[\nabla^2]-\<Q^{-1}x,\nabla\>.$$
The spectrum of $-Tr[\nabla^2]-\<Q^{-1}x,\nabla\>$ is discrete and consists of positive eigenvalues $0\le\lambda_1(\O)\le\lambda_2(\O)\le\dots$ for which the variational formulation \eqref{s6e1} holds, i.e. $\lambda_k(\O)=\lambda_k(\O;H)$. Moreover we set $\lambda(\O):=\lambda(\O;W^{1,2}(\HH))$ as defined in \eqref{s6e3}.

\begin{teo}
Suppose that $\HH$ is a separable Hilbert space with Gaussian measure $\mu$. Then, for any $0\le c\le1$, the following optimization problem has a solution:
$$\min\Big\{\Phi(\lambda(\O))\ :\ \O\subset X\hbox{ quasi-open, }\mu(\O)=c\Big\},$$
where $\Phi:[0,+\infty]\to\R$ is a functional satisfying the conditions $(\Phi1)$ and $(\Phi2)$ from Section \ref{s4}.
\end{teo}
\begin{proof}
Take $H:=W^{1,2}(\HH)$ and $Du=\|\nabla u\|_{\HH}$. The pair $(H,D)$ satisfies the hypothesis $H1,\dots,\HH3$ and $\HH4$. In fact, the norm $\|u\|^2=\|u\|^2_{L^2}+\|Du\|^2_{L^2}$ is the usual norm in $W^{1,2}(\HH)$ and with this norm $W^{1,2}(\HH)$ is a separable Hilbert space and the inclusion $H\hookrightarrow L^2(\HH,\mu)$ is compact (see \cite[Theorem 9.2.12]{daprato}). Moreover, the continuous functions are dense in $W^{1,2}(\HH)$, by construction. Applying Theorem \ref{main} we obtain the conclusion.
\end{proof}

%%%%%%%%%%%%%%%%%%%%%%%%%%%%%%%%%%%%%%%%%%%%%%%%%%
\subsection{Carnot-Caratheodory spaces}

Consider a bounded open and connected set $\Dr\subset \R^d$ and $C^{\infty}$ vector fields $Y_1,\dots,Y_k$ on $\Dr$. Suppose that the vector fields satisfy the H\"ormander's condition, i.e. the Lie algebra generated by $Y_1,\dots,Y_k$ has dimension $d$ in each point $x\in\O$. Following \cite{smetp} we define a distance on $\Dr$ in the following way:

\begin{deff}
We say that an absolutely continuous curve $\gamma:[a,b]\to\Dr$ is admissible, if there exist measurable functions $c_1,\dots,c_k:[a,b]\to\R$ such that
$$\sum_{j=1}^{k}|c_j(t)|^2\le1\quad\forall t\in[a,b]\quad\hbox{and}\quad\dot{\gamma}(t)=\sum_{j=1}^{k}c_{j}(t)Y(\gamma(t)).$$
The Carnot-Caratheodory distance between $x,y\in\Dr$ with respect to the vector fields $Y_1,\dots,Y_k$ is given by
$$\rho(x,y)=\inf\big\{T>0:\ \exists\gamma:[a,b]\to\Dr\ \hbox{admissibles with}\ \gamma(a)=x,\ \gamma(b)=y\big\}.$$
\end{deff} 

Note that $\rho$ is a distance on $\Dr$ since, in our case, there is always an admissible curve connecting $x$ and $y$. This is a direct consequence from a result due to Sussmann, \cite{sussmann} (for more references and deeper discussion on this topic see \cite{smetp}).

Consider the metric space $(\Dr,\rho)$ equipped with the $d$-dimensional Lebesgue measure $\LL^d$. We define the Sobolev space on $\O$ with respect to the family of vector fields $Y=(Y_1,\dots,Y_k)$ as 
$$W^{1,2}_Y(\Dr)=\{u\in L^2(\Dr):\ Y_ju\in L^2,\ \forall j=1,\dots,k\},$$
$$\|u\|_{1,2}=\left(\|u\|_{L^2}^2+\sum_{j=1}^k \|Y_j u\|_{L^2}^2\right)^{1/2},$$
where the derivation $Y_ju$ is intended in sense of distributions.
For $u\in W^{1,2}_Y(\Dr)$, we define the gradient $Yu=(Y_1u,\dots,Y_ku)$ and set $|Yu|=\left(|Y_1u|^p+\dots+|Y_ku|^{2}\right)^{1/2}\in L^2(\Dr)$.
If $u\in C^{\infty}(\Dr)$, then $|Yu|$ is an upper gradient for $u$ with respect to the distance $\rho$. In fact, if $\gamma$ is a $\rho$-Lipschitz curve, then by Proposition 11.4 of \cite{smetp}, it is admissible and
$$|u(\gamma(b))-u(\gamma(a))|\le\int_a^b\sum_{j=1}^k|c_j(t)(Y_ju)(\gamma(t))|\,dt\le\int_{a}^b |(Yu)(\gamma(t))|dt.$$
Setting $Du=|Yu|$ and $H=W^{1,2}_Y(\Dr)$, we can define the energy $E(\O):=E(\O;H)$ and the spectrum $\lambda(\O)=\lambda(\O;H)$ as in \eqref{s6e1} and \eqref{s6e3}. Below, we obtain an existence result for the functionals of the type $\Phi(\lambda(\O))$, simply by applying Corollary \ref{corcp}. To prove that we are really in the setting of Corollary \ref{corcp}, we start by noting that the set $W^{1,2}_Y(\Dr)\cap C^{\infty}(\O)$ is dense in $u\in W^{1,2}_Y(\Dr)$ (see \cite[Theorem 11.9]{smetp}). Thus, we have that $W^{1,2}_Y(\Dr)$ is a subset of the Cheeger space $H^{1,2}(\Dr,\rho,\lambda)$ and that $|Yu|$ is a weak upper gradient for $u$. In \cite[Theorem 11.7]{smetp} it was shown that it is, actually, the least upper gradient of $u$. By the result of Nagel, Stein and Wainger (see \cite{nastwa}), the Lebesgue measure is doubling with respect to the distance $\rho$. Moreover, the weak Poincar\`e inequality holds on the space $(\Dr,\rho,\LL^d)$ (see \cite{smetp}). Thus we can apply Corollary \ref{corcp}, obtained for metric measure spaces, in the setting of the Carnot-Caratheodory spaces:

\begin{teo}
Consider a family $Y=(Y_1,\dots,Y_k)$ of $C^{\infty}$ vector fields defined on an open neighborhood of the closure of the open connected set $\O\subset\R^d$. If $Y_1,\dots,Y_k$ satisfy the H\"ormander condition. If $\Dr$ is of finite Lebesgue measure and has finite diameter with respect to the Carnot-Caratheodory distance, then for any $0\le c\le|\Dr|)$, the following shape optimization problems admit solutions:
$$\min\big\{\Phi(\lambda(\O))\ :\ \O\subset\Dr\ \O\hbox{ quasi-open, }|\O|\le c\big\},$$
$$\min\big\{E(\O):A\subset\Dr\ \O\hbox{ quasi-open, }|\O|\le c\big\},$$
where $\Phi:[0,+\infty]^\N\to\R$ satisfies the assumptions $(\Phi_1)$ and $(\Phi_2)$ from Section \ref{s4}, $\lambda(\Dr)=\lambda(\Dr;W^{1,2}_Y)$ and $E(\Dr)$ are as in \eqref{s6e1} and \eqref{s6e3}.
\end{teo}

\bigskip
\begin{ack}
The authors wish to thank Luigi Ambrosio and Giuseppe Da Prato for some useful discussions and for reading parts of the paper in the early stages of its preparation. This work is part of the project 2008K7Z249 {\it``Trasporto ottimo di massa, disuguaglianze geometriche e funzionali e applicazioni''} financed by the Italian Ministry of Research.
\end{ack}

%%%%%%%%%%%%%%%%%%%%%%%%%%%%%%%%%%%%%%%%%%%%%%%%%%

\bigskip
{\small
\begin{minipage}[t]{6.9cm}
Giuseppe Buttazzo\\
Dipartimento di Matematica\\
Universit\`a di Pisa\\
Largo B. Pontecorvo, 5\\
56127 Pisa - ITALY\\
{\tt buttazzo@dm.unipi.it}
\end{minipage}
\begin{minipage}[t]{6.9cm}
Bozhidar Velichkov\\
Scuola Normale Superiore di Pisa\\
Piazza dei Cavalieri, 7\\
56126 Pisa - ITALY\\
{\tt b.velichkov@sns.it}
\end{minipage}}

\end{document}